\documentclass[12pt]{amsart}

\usepackage{pdfsync}         
\usepackage[active]{srcltx}  
\usepackage{enumerate}
\usepackage{comment}
\usepackage{amssymb}
\usepackage{subcaption}
\usepackage{mathtools}
\usepackage{enumitem}

\newenvironment{labeledlist}[2][\unskip]
{ \renewcommand{\theenumi}{\textnormal{({#2}\arabic{enumi}{#1})}}
  \renewcommand{\labelenumi}{\textnormal{({#2}\arabic{enumi}{#1})}}
  \begin{enumerate} }
{ \end{enumerate} }

\usepackage[OT1]{fontenc}    %
\usepackage{graphicx}

\usepackage{palatino}
\usepackage[left=3cm,top=3.8cm,right=3cm]{geometry}        

\usepackage{xcolor}
\usepackage[pagebackref]{hyperref}
\hypersetup{
   colorlinks,
    linkcolor={red!60!black},
    citecolor={blue!60!black},
    urlcolor={blue!90!black}
}

\numberwithin{equation}{section}

\theoremstyle{plain}
\newtheorem{theorem}{Theorem}[section]
\newtheorem{corollary}[theorem]{Corollary}
\newtheorem{prop}[theorem]{Proposition}
\newtheorem{lemma}[theorem]{Lemma}

\theoremstyle{remark}

\theoremstyle{definition}
\newtheorem{definition}[theorem]{Definition}

\newcommand{\lam}{\lambda}
\newcommand{\1}{\mathbf{1}}

\newcommand{\e}{\varepsilon}
\newcommand{\N}{\mathbb{N}}
\newcommand{\R}{\mathbb{R}}
\newcommand{\Z}{\mathbb{Z}}

\newcommand{\eps}{\varepsilon}

\DeclareMathOperator{\hdim}{dim_H}

\DeclareMathOperator{\sdim}{dim_S}

\DeclareMathOperator{\ubdim}{\overline{\dim}_B}

\DeclareMathOperator{\supp}{supp}

\newcommand{\wt}{\widetilde}

\title[$L^q$ dimensions of self-similar measures]{$L^q$ dimensions of self-similar measures, and applications: a survey}

\author{Pablo Shmerkin}

\address{Department of Mathematics and Statistics, Torcuato Di Tella University, and CONICET, Buenos Aires, Argentina}
\email{pshmerkin@utdt.edu}
\urladdr{http://www.utdt.edu/profesores/pshmerkin}
\thanks{Partially supported by Projects CONICET-PIP 11220150100355 and PICT 2015-3675 (ANPCyT)}

\subjclass[2010]{Primary: 28A75, 28A80}
\keywords{Self-similar measures, Bernoulli convolutions, $L^q$ dimensions, Cantor sets, intersections}

\begin{document}

\begin{abstract}
We present a self-contained proof of a formula for the $L^q$ dimensions of self-similar measures on the real line under exponential separation (up to the proof of an inverse theorem for the $L^q$ norm of convolutions). This is a special case of a more general result of the author from [Shmerkin, Pablo. On Furstenberg's intersection conjecture, self-similar measures, and the $L^q$ norms of convolutions. Ann. of Math., 2019], and one of the goals of this survey is to present the ideas in a simpler, but important, setting. We also review some applications of the main result to the study of Bernoulli convolutions and intersections of self-similar Cantor sets.
\end{abstract}

\maketitle

\section{Introduction}

\subsection{Self-similar measures}

The purpose of this survey is to present a special, but important, case of the main result of \cite{Shmerkin19} concerning the smoothness properties of self-similar measures on the real line. Given a finite family $f_i(x)=\lambda_i x +t_i$, $i\in I$ of contracting similarities (that is, $|\lambda_i|<1$) and a corresponding probability vector $(p_i)_{i\in I}$ there is a unique Borel probability measure $\mu$ on $\R$ such that
\[
\mu = \sum_{i\in I} p_i\, f_i\mu,
\]
where here and throughout the paper, if $\nu$ is a Borel probability measure on a space $X$ and $g:X\to Y$ is a Borel map, then $g\nu$ is the push-forward measure, that is, $g\nu(B)=\nu(g^{-1}B)$ for all Borel $B$. We call the tuple $(f_i,p_i)_{i\in I}$ a \emph{weighted iterated function system}, or WIFS, and $\mu$ the corresponding \emph{invariant self-similar measure}.

Studying the properties of self-similar measures, and in particular quantifying their smoothness, is a topic of great interest since the 1930's. Let us define the \emph{similarity dimension} of a self-similar measure $\mu$ (or, rather, the generating WIFS) by
\[
\sdim(\mu) = \frac{\sum_{i\in I} p_i \log(1/p_i)}{\sum_{i\in I} p_i \log(1/\lambda_i)}.
\]
The similarity dimension is one of the simplest instances of a very widespread expression in the dimension theory of conformal dynamical systems: it has the form
\[
\frac{\text{entropy}}{\text{Lyapunov exponent}}.
\]
Indeed, $\sum_{i\in I} p_i \log(1/p_i)$ is the entropy of the probability vector $(p_i)_{i\in I}$: a quantity that measures how uniform this vector is. For example, it attains its maximal value $\log |I|$ exactly at the uniform probability vector $(1/|I|,\ldots,1/|I|)$. Lyapunov exponents quantify the average expansion or contraction of a dynamical system, and this is how the denominator $\sum_{i\in I} p_i \log(1/\lambda_i)$ should be interpreted.

It is well known that if $\sdim(\mu)<1$, then $\mu$ is purely singular with respect to Lebesgue measure. In fact, even more is true. The Hausdorff dimension of a Radon measure $\nu$ on $\R$ is defined as
\[
\hdim(\nu) = \inf\{ \hdim(A): A \text{ is Borel }, \nu(\R\setminus A)=0\}.
\]
For self-similar measures it always holds that $\sdim(\mu)\le \hdim(\mu)$, and it is clear from the definition that measures with Hausdorff dimension $<1$ must be purely singular. We also note that one always has $\hdim(\mu)\le 1$, even though it is possible that $\sdim(\mu)>1$.

Two major problems in fractal geometry are (a) understanding when one actually has $\hdim(\mu)=\sdim(\mu)$, (b) analyzing the properties of $\mu$ when $\sdim(\mu)>1$; in particular, determining whether $\mu$ is absolutely continuous and, if so, characterizing the smoothness of its density. While both problems are still wide open in this generality, major progress has been accomplished in the last few years. The goal of this article is to present one of the several directions in which progress was achieved, following \cite{Shmerkin19}. While the results of \cite{Shmerkin19} concern a wider class of measures satisfying a generalized notion of self-similarity, here we focus on the proper self-similar case, both because it is an important class in itself and because it allows us to present some of the proofs of \cite{Shmerkin19} in a technically simpler setting. In particular, some ergodic-theoretic concepts and tools are not required in the self-similar case.

\subsection{The overlaps conjecture and Hochman's theorem on exponential separation}

Let $A=\supp(\mu)$. The set $A$ is \emph{self-similar}: it satisfies that $A=\bigcup_{i\in I} f_i(A)$ (here we assume that all the $p_i$ are strictly positive; we can always remove the maps $f_i$ with $p_i=0$ to achieve this). When the pieces $f_i(A)$ are separated enough one does have an equality $\hdim(\mu)=\sdim(\mu)$. Indeed, this holds if the sets $f_i(A)$ are pairwise disjoint or, more generally, under the famous open-set condition which allows the images $f_i(A)$ to intersect but in a very limited way.

On the other hand, there are two known mechanisms that force an inequality $\hdim(\mu)<\sdim(\mu)$. The first is if $\sdim(\mu)>1$. The second is slightly less trivial but still quite simple. Suppose first that $f_i=f_j$ for some $i\neq j$. Then if we drop $f_j$ from the WIFS and replace $p_i$ by $p_i+p_j$ the invariant measure does not change. However a simple calculation reveals that the similarity dimension of the new WIFS is strictly smaller than that of the original one, and hence the Hausdorff dimension of $\mu$ is strictly smaller than the similarity dimension derived from the original WIFS. Although the calculation is slightly more involved, the same argument shows that if the maps $f_i$ do not freely generate a free semigroup or, in other words, if there exist different finite sequences $i=(i_1\ldots i_k), j=(j_1\ldots j_\ell)$ such that
\[
f_{i_1}\circ\cdots\circ f_{i_k} =  f_{j_1}\circ\cdots \circ f_{j_\ell},
\]
then one also has $\hdim(\mu)<\sdim(\mu)$. In this case we say that the WIFS has an \emph{exact overlap}. We note that if this happens then it also happens for sequences of the same length, as we could replace $i$ and $j$ by the juxtapositions $ij$ and $ji$.

A central conjecture in fractal geometry asserts that these are the \emph{only} mechanisms by which a dimension drop $\hdim(\mu)<\sdim(\mu)$ can occur (we note that in higher dimensions this is not true, but there are related conjectures, see \cite{Hochman17} for a discussion). This can be shortly stated in the form: if $\hdim(\mu)<\min(\sdim(\mu),1)$, then there is an exact overlap. The conjecture has a long history. A version for sets was stated in print in \cite{PeresSolomyak98}, where it is attributed to K. Simon. We refer to M. Hochman's paper \cite{Hochman14} for further background and discussion.

While the overlaps conjecture remains open, in \cite{Hochman14} M. Hochman accomplished a decisive step towards it. Given a finite sequence $i=(i_1\ldots i_k)$ we write $f_i=f_{i_1}\circ\cdots\circ f_{i_k}$ for short. Roughly speaking, Hochman proved that if $\hdim(\mu)<\min(\sdim(\mu),1)$ then for all large $k$ there must exist distinct pairs $i,j$ of words of length $k$ such that the maps $f_i$ and $f_j$ are super-exponentially close (as opposed to being identical, as the overlap conjecture predicts). More precisely, given two similarity maps $g_j(x)=\lambda_j x+ t_j$, $j=1,2$ we define a distance
\[
d(g_1,g_2) = \left\{
\begin{array}{ccc}
  |t_1-t_2| & \text{ if } & \lambda_1=\lambda_2 \\
  1 & \text{ if } & \lambda_1\neq \lambda_2
\end{array}
\right..
\]
(It may seem strange to define the distance to be $1$ if $\lambda_1\approx\lambda_2$ and $t_1\approx t_2$, but it turns out that only the case in which $\lambda_1=\lambda_2$ ends up being relevant.) Given a WIFS as above, we define the $k$-separation numbers $\Gamma_k$ as
\begin{equation} \label{eq:def-sep-k}
\Gamma_k =  \inf \{ d(g_i,g_j) : i=(i_1,\ldots,i_k), j=(j_1,\ldots,j_k), i\neq j \}.
\end{equation}
We say that the WIFS has \emph{exponential separation} if there exists $\delta>0$ such that
\[
\Gamma_k > \delta^k \quad\text{for infinitely many } k\in \N.
\]
Note that this notion depends only on the similarity maps $f_i$ and not on the weights $p_i$. We also remark that this condition is substantially weaker than the open set condition. We can now state Hochman's Theorem.
\begin{theorem} \label{thm:Hochman}
If $(f_i,p_i)$ is a WIFS with exponential separation and $\mu$ is the associated invariant self-similar measure, then
\[
\hdim(\mu) = \min(\sdim(\mu),1).
\]
\end{theorem}
Besides conceptually getting us closer to the overlaps conjecture, this theorem has some striking implications: it can be checked in many new explicit cases, and it can be shown to hold outside of very small exceptional sets of parameters in parametrized families of self-similar measures satisfying minimal regularity and non-degeneracy assumptions. Hochman's Theorem (and its proof) has also underpinned much of the more recent progress in the area - we will come back to all these points in Section \ref{sec:applications}.

\subsection{$L^q$ dimensions}

Hochman's Theorem is about the Hausdorff dimension of self-similar measures. In fractal geometry, and in particular in multifractal analysis, there is a myriad of other ways of quantifying the size of a (potentially fractal) measure. Of particular relevance is a one-dimensional family of numbers known as \emph{$L^q$ dimensions}, which we now define.

We introduce some further notation for simplicity. Let $\mathcal{P}$ denote the family of boundedly supported Borel probability measures on $\R$. The class of $\mu\in\mathcal{P}$ supported on $[0,1)$ is denoted by $\mathcal{P}_1$. Given $m\in\N$, we let $\mathcal{D}_m$ denote the family of half-open dyadic intervals of side-length $2^{-m}$, that is,
\[
\mathcal{D}_m = \{ [j 2^{-m}, (j+1) 2^{-m}) : j\in\Z\}.
\]
Given $\mu\in\mathcal{P}$ and $q>1$, the quantity $S_m(\mu,q)=\sum_{J\in\mathcal{D}_m} \mu(J)^q$ measures, in an $L^q$-sense, how uniformly distributed $\mu$ is at scale $2^{-m}$. Using H\"{o}lder's inequality, one can check that if $\mu\in\mathcal{P}_1$, then
\begin{equation}  \label{eq:inequalities-moment-sums}
2^{(1-q)m} \le   S_m(\mu,q) \le 1,
\end{equation}
with the extreme values attained, respectively, when $\mu$ is uniformly distributed among the $2^m$ intervals $J\in\mathcal{D}_m$ contained in $[0,1)$, and when $\mu$ gives full mass to a single interval in $\mathcal{D}_m$. This suggests that the decay rate of $S_m(\mu,q)$ as $m\to\infty$ may indicate the smoothness of $\mu$ at arbitrarily small scales, and this is precisely how the $L^q$ dimensions $D_\mu(q)$ are defined:
\begin{align*}
\tau_\mu(q) &=  \liminf_{m\to\infty} \frac{-\log S_m(\mu,q)}{m},\\
D_\mu(q) &= \frac{\tau_\mu(q)}{q-1}.
\end{align*}
(Here and throughout the paper, the logarithms are to base $2$.) We will sometimes write $\tau(\mu,q)$, $D(\mu,q)$ instead of $\tau_\mu(q)$, $D_\mu(q)$. The function $q\mapsto\tau_\mu(q)$ is called the \emph{$L^q$ spectrum} of $\mu$.   In light of \eqref{eq:inequalities-moment-sums}, one always has $0\le D_\mu(q)\le 1$ for $\mu\in\mathcal{P}_1$ and, indeed, the same inequality holds for $\mu\in\mathcal{P}$. Moreover, $D_\mu(q)=0$ for purely atomic measures $\mu$ and $D_\mu(q)=1$ if $\mu$ is Lebesgue measure on an interval or, more generally, if $\mu$ is absolutely continuous with an $L^q$ density. These basic properties suggest that $D_\mu(q)$ is a reasonable notion of dimension.

We state two simple and well known properties of $L^q$ dimensions.
\begin{lemma}
The functions $q\mapsto D_\mu(q)$, $q\mapsto \tau_\mu(q)$ are respectively  non-increasing and concave on $(1,\infty)$.
\end{lemma}
\begin{proof}
Fix $0<\lam<1$, $\mu\in\mathcal{P}$, $m\in\N$, $q_1,q_2\ge 1$. It follows from H\"{o}lder's inequality applied with exponents $1/\lam$ and $1/(1-\lam)$ that
\[
S_m(\mu,\lam q_1+(1-\lam)q_2) \le S_m(\mu,q_1)^\lam \, S_m(\mu,q_2)^{1-\lam}.
\]
The concavity of $\tau$ is immediate from this. For the monotonicity of $D_\mu(q)$, suppose $1<p<q$ and apply the above with $q_1=q,q_2=1$ and $\lam=(p-1)/(q-1)$.
\end{proof}

So far we have dealt with a general measure $\mu\in\mathcal{P}$. We now turn to self-similar measures associated to a WIFS $(f_i,p_i)$. We have seen that the similarity dimension is a ``candidate'' for the Hausdorff dimension of a self-similar measure $\mu$, is always an upper bound for $\hdim(\mu)$ and is conjectured to equal $\hdim(\mu)$ under the terms of the overlaps conjecture. There is a natural $L^q$ analog of the similarity dimension: first, we define $T(\mu,q)$ as the only number satisfying
\[
\sum_{i\in I} p_i^q \lambda_i^{-T(\mu,q)} = 1,
\]
and then let $\sdim(\mu,q)=T(\mu,q)/(q-1)$. The function $T$ is a ``symbolic'' analog of the $L^q$ spectrum, while $\sdim$ is a version of similarity dimension for $L^q$ dimensions.

A simple exercise shows that $\lim_{q\to 1+} \sdim(\mu,q)=\sdim(\mu)$. We also have that $q\mapsto \sdim(\mu,q)$ is a real-analytic, nondecreasing function of $q$; it is constant if and only if $p_i =\lambda_i^s$ for some $s$ independent of $i$ (in which case $\sdim(\mu,q)=s$ for all $q>1$), and otherwise it is strictly decreasing.

Just like for Hausdorff dimension, it always holds that $D_\mu(q)\le \min(\sdim(\mu,q),1)$, and the only known mechanisms for a strict inequality are $\sdim(\mu,q)>1$ and the presence of exact overlaps. A variant of the overlaps conjecture asserts that if $\mu$ is a self-similar measure then $D_\mu(q)=\min(\sdim(\mu,q),1)$ unless there is an exact overlap. This conjecture is stronger than the Hausdorff dimension variant, since $D_\mu(q)\le \hdim(\mu)$ for all $q>1$ and $D_\mu(q)\to \hdim(\mu)$ as $q\to 1^+$ in the case of self-similar measures: see \cite[Theorem 5.1 and Remark 5.2]{ShmerkinSolomyak16}.

In \cite{Shmerkin19}, the author established the following variant of Hochman's Theorem \ref{thm:Hochman} for $L^q$ dimensions:
\begin{theorem} \label{thm:main}
If $(f_i,p_i)$ is a WIFS with exponential separation and $\mu$ is the associated invariant self-similar measure, then
\[
D_\mu(q) = \min(\sdim(\mu,q),1) \quad\text{for all } q>1.
\]
\end{theorem}
Again, this theorem is formally stronger than Theorem \ref{thm:Hochman}, since the latter can be recovered by letting $q\to 1^+$. While at first it may seem that the difference between Hausdorff and $L^q$ dimensions is merely technical, the $L^q$ dimension version has several advantages in applications, especially since it applies to \emph{every} $q>1$. It is useful to think of the difference between $L^q$ and Hausdorff dimensions as being similar to the difference between $L^q$ and $L^1$ functions (on bounded intervals). Both kind of dimensions give information about the local behaviour of a measure, but the $L^q$ dimensions do so in a more quantitative fashion. If $\hdim(\mu)>s$, then it holds that
\begin{equation} \label{eq:Frostman}
\mu(B(x,r)) \le r^s
\end{equation}
for $\mu$-almost all $x$ and all sufficiently small $r$ (depending on $x$). On the other hand, if $\lim_{q\to\infty} D_\mu(q)>s$, then \eqref{eq:Frostman} holds \emph{uniformly, for all $x$} and all sufficiently small $r$: see Lemma \ref{lem:Lq-dim-to-Frostman-exp} below. For some applications of Theorem \ref{thm:main} beyond those described in this article, see \cite{BRS19, FraserJordan17, RossiShmerkin18}.

Theorem \ref{thm:main} was originally featured in \cite[Theorem 6.6]{Shmerkin19}. In this article we will present the proof of the special case in which the WIFS is \emph{homogeneous}, that is, all of the scaling factors $\lambda_i$ are equal. The homogeneous case of Theorem \ref{thm:main} is a particular case of \cite[Theorem 1.1]{Shmerkin19}. As indicated earlier, this particular case avoids an ergodic-theoretic part of the argument, and so we hope it will be more accessible. The proof borrows many ideas from Hochman's proof of Theorem \ref{thm:Hochman}, but there are also substantial differences.

One central element of the proof of Theorem \ref{thm:main} is an inverse theorem for the $L^q$ norm of convolutions, which does not rely on self-similarity and may have other applications. This theorem is discussed and stated (without proof) in Section \ref{sec:inverse-theorem}. Section \ref{sec:proof} contains the proof of the homogeneous version of Theorem \ref{thm:main}, starting with a sketch and proceeding to the details. In Section \ref{sec:applications} we introduce some applications to Frostman exponents, self-similar measures generated by algebraic parameters, absolute continuity and intersections of self-similar Cantor sets. We also briefly discuss some old and new results by other authors on Bernoulli convolutions, how they relate to ours, and a possible line for future research.

Although some of the applications in Section \ref{sec:applications} have not been stated in this form in \cite{Shmerkin14}, both the results and the presentation of this survey are strongly based on \cite{Shmerkin14}.

\section{An inverse theorem for the $L^q$ norms of convolutions}

\label{sec:inverse-theorem}

Let $\mu,\nu\in \mathcal{P}$. The convolution $\mu*\nu$ is defined as the push-forward of the product measure $\mu\times\nu$ under the addition map $(x,y)\mapsto x+y$. Explicitly,
\[
\mu*\nu(A) = (\mu\times\nu)\{(x,y):x+y\in A\}  \quad\text{for all Borel }A\subset\R.
\]
Intuitively, one expects the convolution $\mu*\nu$ to be at least as smooth as $\mu$. A natural question is then: if $\mu*\nu$ is not ``much smoother'' than $\mu$, can we deduce any structural information about the measures $\mu$ and $\nu$? Of course, this depends on the notion of smoothness under consideration, and on the precise meaning of ``much smoother''.

Here we will measure smoothness by the moment sums $S_m(\mu,q)$ (with $q>1$ fixed, and $m$ also fixed but very large). Nevertheless, we begin by discussing the situation for entropy. Let $\mu\in\mathcal{P}_1$. Its normalized level $m$ entropy is
\[
H_m(\mu) = \frac{1}{m}\sum_{J\in\mathcal{D}_m} -\mu(J)\log(\mu(J)),
\]
with the usual convention $0\log 0=0$.  In \cite[Theorem 2.7]{Hochman14}, Hochman showed that if $\mu,\nu\in\mathcal{P}_1$ satisfy
\[
H_m(\nu*\mu) \le H_m(\mu)+\e,
\]
where $\e>0$ is small, then $\nu$ and $\mu$ have a certain structure which, very roughly, is of this form: the set of dyadic scales $0\le s<m$ can be split into three sets $\mathcal{A}\cup \mathcal{B}\cup \mathcal{C}$. At scales in $\mathcal{A}$, the measure $\nu$ looks ``roughly atomic'', at scales in $\mathcal{B}$ the measure $\mu$ looks ``roughly uniform'', and the set $\mathcal{C}$ is small. This theorem was motivated in part by its applications to the dimension theory of self-similar measures, as discussed above.

We aim to state a result in the same spirit, but with $L^q$ norms in place of entropy. Given $m\in\N$, we will say that $\mu$ is a $2^{-m}$-measure if $\mu$ is a probability measure supported on $2^{-m}\Z\cap [0,1)$. Given $\mu\in\mathcal{P}_1$, we denote by $\mu^{(m)}$ the associated $2^{-m}$-measure, given by $\mu^{(m)}(j 2^{-m}) = \mu([j 2^{-m},(j+1)2^{-m}))$. Given a purely atomic measure $\rho$ we define the $L^q$ norms
\[
\|\rho\|_q = \left(\sum \rho(y)^q\right)^{1/q},
\]
for $q\in (1,\infty)$ and also set $\|\rho\|_\infty= \max_y \rho(y)$. With these definitions, we clearly have
\[
S_m(\mu,q) = \|\mu^{(m)}\|_q^q.
\]

By the convexity of $t\mapsto t^q$, we have that $\|\mu*\nu\|_q\le \|\mu\|_q\|\nu\|_1$, for any $q\ge 1$ and any two finitely supported probability measures $\mu,\nu$ (this is a simple instance of Young's convolution inequality). We aim to understand under what circumstances $\|\mu*\nu\|_q \approx \|\mu\|_q\|\nu\|_1$, where the closeness is in a weak, exponential sense.  More precisely, we are interested in what structural properties of two $2^{-m}$-measures $\mu,\nu$ ensure an exponential flattening of the $L^q$ norm of the form
\begin{equation} \label{eq:intro-L2-norm-flattens}
\|\mu*\nu\|_q \le 2^{-\e m} \|\mu\|_q.
\end{equation}
(Recall that, by definition, $2^{-m}$-measures are probability measures, so that $\|\nu\|_1=1$.) One particular instance of this problem has received considerable attention. Given a finite set $A$, we denote $\mathbf{1}_A=\sum_{x\in A} \delta_x$.  Then $\|\mathbf{1}_A*\mathbf{1}_A\|_2^2$ is the additive energy of $A$, a quantity of great importance in combinatorics and its applications. In particular, estimates of the form
\[
\|\mathbf{1}_A*\mathbf{1}_A\|_2^2 \le |A|^{-\e} \|\mathbf{1}_A\|_2^2 \|\mathbf{1}_A\|_1 =  |A|^{3-\e}
\]
arise repeatedly in dynamics, combinatorics and analysis: see e.g. \cite{DyatlovZahl16, ALL17} for some recent examples.

To motivate the inverse theorem, we discuss cases in which $\|\mu *\nu\|_q \approx \|\mu\|_q$ for $2^{-m}$-measures $\mu$ and $\nu$, where we are deliberately vague about the exact meaning of $\approx$. If $\nu=\delta_{k 2^{-m}}$, then $\mu*\nu$ is just a translation of $\mu$ and so we have an exact equality. If $\nu$ is supported on a small number of atoms (say subexponential in $m$), then we still have $\|\mu*\nu\|_q \approx \|\mu\|_q$. Reciprocally, if $\lam$ denotes the uniform $2^{-m}$-measure giving mass $2^{-m}$ to each atom $j 2^{-m}$,  then we also have $\|\lam*\nu\|_q \approx \|\lam\|_q$. The same holds if $\lam$ is replaced by a suitably small perturbation.

Furthermore, if $\nu=2^{-\e m}\delta_0 + (1- 2^{\e m})\lambda$ and $\mu$ is an arbitrary $2^{-m}$-measure, then we still have $\|\mu*\nu\|_q \ge 2^{-\e m}\|\mu\|_q$. This shows that a subset of measure $2^{-\e m}$ is able to prevent exponential smoothening, so that in order to guarantee \eqref{eq:intro-L2-norm-flattens} we need to impose conditions on the structure of the measures inside sets of exponentially small measure. This is one significant difference with the case of entropy, since sets of exponentially small measure have negligible contribution to the entropy.

A naive conjecture might be that if \eqref{eq:intro-L2-norm-flattens} fails for a pair of $2^{-m}$-measures, then either $\mu$ is close to uniform, or $\nu$ gives ``large'' mass to an exponentially small set of atoms. However, there other situations in which $\|\mu*\nu\|_q\approx \|\mu\|_q$. Let $D\gg 1$ be a large integer and fix $\ell\gg D$. Given a subset $\mathcal{S}$ of $\{0,\ldots,\ell-1\}$, let $\widetilde{\mu}$ be the distribution of an independent sequence of random variables $(X_1,\ldots, X_\ell)$ such that $X_s$ is uniformly distributed in $\{0,1,\ldots,2^D-1\}$ if $s\in\mathcal{S}$ and $X_s=0$ if $s\notin\mathcal{S}$. Finally, let $\mu$ be the push-forward of $\widetilde{\mu}$ under the $2^D$-ary expansion map. In other words, $\widetilde{\mu}$ is the $2^{-D\ell}$-measure such that
\[
\mu\left(\sum_{s=1}^\ell X_s 2^{-Ds}\right) = \widetilde{\mu}(X_1,\ldots,X_s).
\]
It is convenient to think about the structure of $\mu$ in terms of trees. Given a base $2^D$ and a nonempty closed subset $A$ of $[0,1)$, we may associate to $A$ the family of all intervals of the form $[j 2^{-Ds},(j+1) 2^{-Ds})$ (i.e. the $2^D$-ary intervals) that intersect $A$. This family has a natural tree structure, where the interval $[0,1)$ is the root and descendence is given by inclusion. In general, the tree associated to $A$ is infinite, but in the case of $2^{-D\ell}$-sets we can think of a finite tree with $\ell$ levels. For the measure $\mu$ just defined, its support $A$ has the following structure: vertices of level $s\in\mathcal{S}$ have a maximal number of offspring $2^D$ (``full branching'') while vertices of level $s\notin\mathcal{S}$ have a single offspring (``no branching''), corresponding to the leftmost interval. Moreover, $\mu$ is the uniform measure on $A$ - it gives all points in $A$ the same mass $1/|A|$.

The convolution $\mu*\mu$ has essentially the same structure, except that vertices of level $s\notin\mathcal{S}$ such that $s-1\in\mathcal{S}$ have two offspring - due to the carries of the previous level. Using this structure, it is not hard to check that $\|\mu*\mu \|_q \approx \|\mu\|_q$ for all $q$.  In similar ways one can construct $2^{-m}$-measures $\mu,\nu$ supported on sets of widely different sizes, such that $\|\mu*\nu\|_q\approx \|\mu\|_q$.

The inverse theorem asserts that if \eqref{eq:intro-L2-norm-flattens} fails to hold then one can find subsets $A\subset \supp(\mu)$ and $B\subset\supp(\nu)$, such that $A$ captures a ``large'' proportion of the $L^q$ norm of $\mu$ and $B$ a ``large'' proportion of the mass of $\nu$, and moreover $\mu|_A,\nu|_B$ are fairly regular (they are constant up to a factor of $2$). The main conclusion, however, is that $A$ and $B$ have a structure resembling the example above, and also the conclusion of Hochman's inverse theorem for entropy: if $D$ is a large enough integer, then for each $s$, either $B$ has no branching between scales $2^{-sD}$ and $2^{-(s+1)D}$ (in other words, once the first $s$ digits in the $2^D$-ary expansion of $y\in B$ are fixed, the next digit is uniquely determined), or $A$ has nearly full branching between scales $2^{-sD}$ and $2^{-(s+1)D}$ (whatever the first $s$ digits of $x\in A$ in the $2^D$-adic expansion, the next digit can take ``most'' values). To formalize this, we introduce the following definition:

\begin{definition}
Given $D\in\N, \ell\in\N$ and a sequence $R=(R_0,\ldots,R_{\ell-1})\in [1,2^D]^\ell$,  we say that a set $A\subset [0,1)$ is \emph{$(D,\ell,R)$-regular} if it is a $2^{-\ell D}$-set,  and for all $s\in \{0,\ldots,\ell-1\}$ and for all $J\in\mathcal{D}_{sD}$ such that $A\cap J\neq\varnothing$, it holds that
\[
| \{ J'\in\mathcal{D}_{(s+1)D} : J'\subset J, J'\cap A \neq\varnothing \} | = R_s.
\]
\end{definition}
In terms of the associated $2^D$-ary tree, $A$ is $(D,\ell,R)$-regular if every vertex of level $s$ has the same number of offspring $R_s$.

Before stating the theorem, we summarize our notation for dyadic intervals  (some of it has already been introduced):
\begin{itemize}
\item $\mathcal{D}_s$ is the family of dyadic intervals $[j 2^{-s},(j+1) 2^{-s})$.
\item Given a set $A\subset\R$, we write $\mathcal{D}_s(A)$ for the family of intervals in $\mathcal{D}_s$ that hit $A$.
\item Given $x\in\R$, we write $\mathcal{D}_s(x)$ for the only interval in $\mathcal{D}_s$ that contains $x$.
\item We write $a J$ for the interval of the same center as $J$ and length $a$ times the length of $J$.
\end{itemize}
We also write $[\ell]=\{0,1,\ldots,\ell-1\}$.

\begin{theorem}
\label{thm:inverse-thm}
For each $q>1$, $\delta>0$, and $D_0\in\N$, there are $D \geq D_0$ and $\eps > 0$, so that the following holds for $\ell\ge \ell_0(q,\delta,D_0)$.

Let $m=\ell D$ and let $\mu$ and $\nu$ be $2^{-m}$-measures with
\[
 \|\mu \ast \nu\|_q \geq 2^{-\eps m} \|\mu\|_q.
\]
Then there exist $2^{-m}$-sets $A\subset \supp\mu$ and $B\subset\supp\nu$, numbers $k_A, k_B\in 2^{-m}\Z$, and a set $\mathcal{S}\subset[\ell]$, so that
\begin{labeledlist}{A}
 \item\label{A1} $\| \mu|_A \|_q \geq 2^{-\delta m}\|\mu\|_q$.
 \item\label{A2} $\mu(x) \leq 2 \mu(y)$ for all $x,y\in A$.
 \item\label{A3} $A'=A+k_A$ is contained in $[0,1)$ and is $(D,\ell,R')$ uniform for some sequence $R'$.
 \item\label{A4} $x\in \frac{1}{2}\mathcal{D}_{sD}(x)$ for each $x\in A'$ and $s\in [\ell]$.
\end{labeledlist}
\begin{labeledlist}{B}
 \item\label{B1} $\| \nu|_B \|_1 = \nu(B) \geq 2^{-\delta m}$.
 \item\label{B2} $\nu^{(m)}(x) \leq 2 \nu(y)$ for all $x,y\in B$.
 \item\label{B3} $B'=B+k_B$ is contained in $[0,1)$ and is $(D,\ell,R'')$ uniform for some sequence $R''$.
 \item\label{B4} $y\in \frac{1}{2}\mathcal{D}_{sD}(y)$ for each $y\in B'$ and $s\in [\ell]$.
\end{labeledlist}
Moreover
\begin{labeledlist}{}
\setcounter{enumi}{4}
 \item\label{inverse5} for each $s$, $R''_s=1$ if $s\not\in\mathcal{S}$, and $R'_s \geq 2^{(1-\delta)D}$ if $s\in\mathcal{S}$.
 \item\label{inverse6} The set $\mathcal{S}$ satisfies
 \[
  \log \|\nu\|^{-q'}_q - m\delta \leq D|\mathcal{S}| \leq \log \|\mu\|^{-q'}_q + m\delta.
 \]
\end{labeledlist}
\end{theorem}

Here, and throughout the paper, $q'=q/(q-1)$ denotes the dual exponent. We make some remarks on the statement.
\begin{enumerate}[label={\upshape\alph*)}]
\item In the original version of the theorem in \cite{Shmerkin19}, both the convolution and the translations take place on the circle $[0,1)$ with addition modulo $1$. See \cite[Theorem 2.2 and Remark 2.3]{RossiShmerkin18} for this formulation.
\item The main claim in the theorem is part (5). Obtaining sets $A,B$ satisfying (A1)--(A4) and (B1)--(B4) is not hard, and (6) is a straightforward calculation using (5).
\item The theorem fails for $q=1$ and $q=\infty$. In the first case there is an equality $\|\mu*\nu\|_1=\|\mu\|_1$ for any $2^{-m}$-measures, and in the second case there is always an equality $\|\1_A*\1_{-A}\|_\infty= \|\1_A\|_\infty\|\1_A\|_1$. On the other hand, the proof can easily be reduced to the case $q=2$, with the remaining cases following by interpolation with the endpoints $q=1$ and $q=\infty$.
\end{enumerate}

The proof of Theorem \ref{thm:inverse-thm} (including the proofs of the results it relies on) is elementary and, at least in principle, it is effective, although the value of $\e$ that emerges from the proof is extremely poor and certainly sub-optimal. However, for the purposes of proving Theorem \ref{thm:main}, the existence of any $\e>0$ is enough.

\section{Proof of the main theorem}
\label{sec:proof}

\subsection{Homogeneous self-similar measures}

We restate the particular case of Theorem \ref{thm:main} that we will prove.
\begin{theorem} \label{thm:main-homogeneous}
Let $(f_i(x)=\lambda x+t_i)_{i\in I}$ be a homogeneous IFS with exponential separation. Then for any probability vector $(p_i)_{i\in I}$, if $\mu$ is the invariant self-similar measure for the WIFS $(f_i,p_i)$, then
\begin{equation} \label{eq:main-thm}
D_\mu(q) = \min(\sdim(\mu,q),1) = \min\left(\frac{\log \sum_{i\in I} p_i^q}{(q-1)\log\lambda},1\right)
\end{equation}
for all $q>1$.
\end{theorem}
We recall that ``homogeneous'' here refers to the fact that all scaling factors are equal. We may and do assume that $\lambda>0$; if $\lambda<0$, note that $\mu$ can also be generated by the WIFS $(f_i f_j, p_i p_j)_{i,j\in I}$, for which the scaling factor is $\lambda^2>0$. This iteration of the IFS does not change the validity of exponential separation.

The key advantage of homogeneity is that, in this case, the self-similar measure $\mu$ has an infinite convolution structure: if $\Delta= \sum_{i\in I} p_i \delta_{t_i}$, then
\begin{equation} \label{eq:convolution-structure}
\mu = *_{n=0}^\infty S_{\lambda^n}\Delta,
\end{equation}
where $S_a(x)=ax$ rescales by $a$. Formally, this infinite convolution is defined as the push-forward of the countable self-product $\mu^{\N}$ under the series expansion map $(x_1,x_2,\ldots)\mapsto \sum_{n=0}^\infty x_n$; this is well-defined since the series always converges absolutely. To verify that this is indeed the self-similar measure, one only needs to check that it satisfies the self-similarity relation
\[
\mu = \sum_{i\in I} p_i \, f_i\mu.
\]
In more probabilistic terms, $\mu$ can also be defined as the distribution of the random series $\sum_{n=0}^\infty \lambda^n X_n$, where $X_n$ are IID random variables with distribution $\Delta$. The well-known fact that the distribution of a sum of independent random variables is the convolution of the distributions (which extends to countable sums) then gives another derivation of \eqref{eq:convolution-structure}.

\subsection{Outline of the proof}

The overall strategy of the proof of Theorem \ref{thm:main-homogeneous} follows the broad outline of \cite{Hochman14}. However, while Hochman's method is based on entropy, we need to deal with $L^q$ norms and, as we will see, this forces substantial changes in the implementation of the outline.

The right-hand side in \eqref{eq:main-thm} is easily seen to be an upper bound for the left-hand side, so the task is to show the reverse inequality. Write $\tau=\tau_\mu$ and $D=D_\mu$. We want to show that if $D(q)<1$ (or, equivalently, $\tau(q)<q-1$) then $D(q)=\sdim(\mu,q)$ (under the hypothesis of exponential separation).

Recall that the $L^q$ spectrum $\tau(q)$ is concave, so in particular it is continuous and differentiable outside of at most a countable set. Hence it is enough to prove the claim above for a fixed differentiability point $q$. The advantage of this assumption is that the ``multifractal structure'' of a measure $\mu$ is known to behave in a regular way for points $q$ of differentiability of the spectrum. In particular, we will see that if $\alpha=\tau'(q)$ then, for large enough $m$, ``almost all'' of the contribution to the sum $\sum_{J\in\mathcal{D}_m}\mu(J)^q$ comes from $\approx 2^{\tau^*(\alpha)m}$ intervals $I$ such that $\mu(J)\approx 2^{\alpha m}$; here $\tau^*$ is the Legendre transform of $\tau$ (see \S\ref{subsec:multifractal} for the definition). Moreover, using the self-similarity of $\mu$, we establish also a multi-scale version of this fact, see Proposition \ref{prop:Lq-over-small-set-is-small}.

The following is the key estimate in the proof; as we will see in Section \ref{sec:applications}, it has other applications.
Recall that $\mu^{(m)}$ is given by
\begin{equation} \label{eq:def-mu-m}
\mu^{(m)}(j 2^{-m}) = \mu([j2^{-m},(j+1)2^{-m}))
\end{equation}
and that, by definition, $\|\mu^{(m)}\|_q^q= S_m(\mu,q) \approx 2^{-m \tau(q)}$.
\begin{theorem} \label{thm:conv-with-ssm-flattens}
Let $\mu$ be a self-similar measure associated to a homogeneous WIFS (not necessarily with exponential separation) and let $q>1$. Suppose $\tau_\mu(q)<q-1$. Then for every $\sigma>0$ there is $\e=\e(\sigma,q)>0$ such that the following holds for all large enough $m$: if $\rho$ is an arbitrary $2^{-m}$-measure such that $\|\rho\|_q^{q'} \le 2^{-\sigma m}$, then
\begin{equation} \label{eq:outline-conv-flattens-mu}
\|\rho*\mu^{(m)}\|_q^q \le 2^{-\e m} \|\mu^{(m)}\|_q^q.
\end{equation}
\end{theorem}

This theorem is proved by combining the inverse theorem for the $L^q$ norm of convolutions (Theorem \ref{thm:inverse-thm}), together with the study of the multifractal structure of $\mu$. We sketch the idea very briefly: suppose \eqref{eq:outline-conv-flattens-mu} fails. The inverse theorem then asserts that there is a regular subset $A$ of $\supp(\mu^{(m)})$ that captures much of the $L^q$ norm of $\mu^{(m)}$. By part (5) of the inverse theorem, and since $\rho$ is assumed to have exponentially small $L^q$ norm, $A$ must have almost full branching on a positive density set of  scales in a multi-scale decomposition. But $A$ itself does not have full branching (this will follow from the assumption $\tau(q)<q-1$, which rules out $\mu^{(m)}$ having too small $L^q$ norm). So there must also be a positive density set of scales on which $A$ has smaller than average branching. The regularity of the multifractal spectrum discussed above rules this out, since it forces $A$ to have an almost constant branching on almost all scales. For a detailed proof, see \S\ref{subsec:conv-ssm-flattens} below.

The conclusion of the proof of Theorem \ref{thm:main-homogeneous} from \eqref{eq:outline-conv-flattens-mu} tracks fairly closely the ideas of \cite{Hochman14}. One consequence of self-similarity, as expressed by \eqref{eq:convolution-structure}, is that
\[
\mu = \mu_n * S_{\lambda^n} \mu,
\]
where
\begin{equation} \label{eq:def-mu-n}
\mu_n = *_{j=0}^{n-1} S_{\lambda^j} \Delta.
\end{equation}
Note that the atoms of $\mu_n$ are the points of the form $f_{i_1}\circ \cdots \circ f_{i_n}(0)$. The exponential separation assumption implies that there exist $R\in\N$ such that all these atoms are distinct and $\lambda^{Rn}$-separated. Hence for this value of $R$ we have
\[
\frac{\log \|\mu_{n}^{(R m)}\|_q^q}{(q-1)n\log(1/\lam)}  = \frac{\log \|\mu_{n}\|_q^q}{(q-1)n\log(1/\lam)} = \frac{n \log\|\Delta\|_q^q}{(q-1)n\log(1/\lam)},
\]
where $m=m(n)$ is chosen so that $2^{-m}\le \lam^n$ and $2^{-m}\sim \lam^n$. It is easy to see that the right-hand side is equal to the right-hand side of \eqref{eq:main-thm}. Hence, it remains to show that
\begin{equation} \label{eq:outline-Lq-norm-scale-Rm}
\lim_{n\to\infty}\frac{\log \|\mu_{n}^{(R m)}\|_q^q}{n\log(1/\lam)} = \tau(q).
\end{equation}
In other words, we need to show that the $L^q$ norm of $\mu_{n}$ at scale $2^{-m}\approx \lam^n$ (which is easily seen to be comparable to the $L^q$ norm of $\mu$ at scale $2^{-m}$, and hence is $\approx S_m(\mu,q)^{1/q}$) nearly exhausts the $L^q$ norm of $\mu_{n}$ at the much finer scale $2^{-R m}$ which, in turn, equals the full $L^q$ norm of $\mu_{n}$, by the exponential separation assumption.

To show \eqref{eq:outline-Lq-norm-scale-Rm}, we recall that $\mu = \mu_{n}* S_{\lam^n}\mu$, and use this to decompose
\[
\mu^{((R+1) m)} =  \sum_{J\in\mathcal{D}_m} \mu(J)  \wt{\rho}_J * S_{\lam^n}\mu,
\]
where $\wt{\rho}_J$ is the normalized restriction of $\mu_{n}$ to $J$. Since the supports of $\wt{\rho}_J * S_{\lam^n}\mu$ have bounded overlap, it is not hard to deduce that
\[
\|\mu^{((R+1) m)}\|_q^q \approx  \sum_{J\in\mathcal{D}_m} \mu(J)^q \| \rho_J * \mu^{(R m)}\|_q^q,
\]
where $\rho_J  = S_{\lam^{-n}}\wt{\rho}_J$. This is the point where we apply Theorem \ref{thm:conv-with-ssm-flattens}, to conclude that if on the right-hand side above we only add over those $J$ such that $\|\rho_J\|_q \ge 2^{-\sigma q}$, where $\sigma>0$ is arbitrary, then, provided $n$ is large enough depending on $\sigma$, we still capture almost all of the left-hand side. This follows since \eqref{eq:outline-conv-flattens-mu} can be shown to imply that the contribution of the remaining $J$ is exponentially smaller than the left-hand side (incidentally, this is the only step where it is crucial to use that $q>1$). A similar calculation, now with $\mu_{n}^{((R+1)m)}$ in place of $\mu^{((R+1) m)}$ in the left-hand side, then shows that \eqref{eq:outline-Lq-norm-scale-Rm} holds, finishing the proof.

\subsection{Notational conventions}

Throughout this section, $\mu$ denotes a self-similar measure associated to a homogeneous WIFS $\{ \lambda x+t_i\}_{i\in I}$ with weights $(p_i)_{i\in I}$. We do \emph{not} assume exponential until the very end, when we finish the proof of Theorem \ref{thm:main-homogeneous}.  We continue to denote
\[
\Delta = \sum_{i\in I} p_i \, \delta_{t_i}.
\]
Other measures, without any assumptions on self-similarity, will be denoted by $\rho$ and $\nu$, possibly with subindices.

We use Landau's $O(\cdot)$ and related notation: if $X,Y$ are two positive quantities, then $Y=O(X)$ means that $Y\le C X$ for some constant $C>0$, while $Y=\Omega(X)$ means that $X=O(Y)$, and $Y=\Theta(X)$ that $Y=O(X)$ and $X=O(Y)$. If the constant $C$ is allowed to depend on some parameters, these are often denoted by subscripts. For example, $Y=O_q(X)$ means that $Y\le C(q) X$, where $C(q)$ is a function depending on the parameter $q$.

\subsection{Preliminary lemmas}

In this section we collect some standard lemmas for later reference. They are all of the form: bounded overlapping does not affect $L^q$ norms too much. We refer to \cite[Section 4]{Shmerkin19} for the very short proofs.

\begin{lemma} \label{lem:Holder}
Let $(Y,\nu,\mathcal{B})$ be a probability space. Suppose $\mathcal{P},\mathcal{Q}$ are finite families of measurable subsets of $Y$ such that each element of $\mathcal{P}$ can be covered by at most $M$ elements of $\mathcal{Q}$ and each element of $\mathcal{Q}$ intersects at most $M$ elements of $\mathcal{P}$. Then, for every $q\ge 1$,
\[
\sum_{P\in\mathcal{P}} \nu(P)^q  \le M^q \sum_{Q\in\mathcal{Q}} \nu(Q)^{q}
\]
\end{lemma}

\begin{lemma} \label{lem:L-q-norm-almost-disj-supports}
Let $\nu=\sum_{i=1}^\ell \nu_i$, where $\nu_i$ are finitely supported measures on a space $Y$, such that each point is in the support of at most $M$ of the $\nu_i$. Then
\[
\|\nu\|_q^q \le M^{q-1} \sum_{i=1}^\ell \|\nu_i\|_q^q.
\]
\end{lemma}

\begin{lemma} \label{lem:discr-norm-conv-equivalence}
For any $q\in (1,\infty)$, for any $\nu_1,\nu_2\in\mathcal{P}_1$  and any $m\in\N$,
\[
 \|(\nu_1*\nu_2)^{(m)}\|_q^q =\Theta_q(1) \|\nu_1^{(m)}*\nu_2^{(m)}\|_q^q.
\]
\end{lemma}

Recall the definition of $\mu_n$ given in \eqref{eq:def-mu-n}.
\begin{lemma} \label{lem:comparison-mu-m-mu}
For any $q\in (1,\infty)$,
\[
 \|\mu^{(m)}\|_q^q =\Theta_q(1) \|\mu_n^{(m)}\|_q^q .
\]
\end{lemma}

\subsection{Multifractal structure}
\label{subsec:multifractal}

We turn to the multifractal estimates that will be required in the proof of Theorem \ref{thm:conv-with-ssm-flattens}. The Legendre transform plays a key role in multifractal analysis. Given a concave function $\zeta:\R\to\R$, its Legendre transform $\zeta^*:\R\to [-\infty,\infty)$ is defined as
\[
\zeta^*(\alpha) = \inf_{q\in\R} \alpha q -\zeta(q).
\]
It is easy to check that if $\zeta$ is concave and is differentiable at $q$, then
\[
\zeta^*(\alpha)=\alpha q-\zeta(q) \text{ for } \alpha=\zeta'(q).
\]
As indicated earlier, we will establish some regularity of the multifractal structure for those values of $q$ such that $\tau$ is differentiable at $q$.

The next lemma is well known; we include the very short proof for completeness.
\begin{lemma} \label{lem:f-alpha-smaller-than-one}
If $\tau$ is differentiable at $q>1$, $\tau(q)<q-1$, and $\alpha=\tau'(q)$, then $\tau^*(\alpha) \le \alpha < 1$
\end{lemma}
\begin{proof}
Since $\tau(1)=0$ (this is immediate from the definition) and  $\tau(q)<q-1$, we have $(\tau(q)-\tau(1))/(q-1)<1$. On the other hand, as $\tau$ is concave and differentiable at $q$, we must have $\alpha\le (\tau(q)-\tau(1))/(q-1)<1$. Furthermore, $\tau^*(\alpha)\le \alpha\cdot 1-\tau(1)=\alpha$, so the lemma follows.
\end{proof}

The following lemmas illustrate the regularity of the $L^q$ spectrum for values $q$ of differentiability of $\tau$ (or dually, points of strict concavity of $\tau^*$). The proofs are similar to  \cite[Theorem 5.1]{LauNgai99}. The heuristic to keep in mind is that, whenever $\alpha=\tau'(q)$ exists, almost all of the contribution to $\|\mu^{(m)}\|_q^q$ comes from $\approx 2^{\tau^*(\alpha) m}$ intervals, each of mass $\approx 2^{-\alpha m}$

\begin{lemma} \label{lem:size-set-A-in-terms-of-f-alpha}
Suppose that $\alpha_0=\tau'(q_0)$ exists for some $q_0\in (1,\infty)$.

Given $\e>0$, the following holds if $\delta$ is small enough in terms of $\e, q_0$ and $m$ is large enough in terms of $\e, q_0$ and $\delta$.

Suppose $\mathcal{D}'\subset\mathcal{D}_m$ is such that
\begin{enumerate}
\item[\textup{(1)}]  $2^{-\alpha m}\le\mu(J)\le 2\cdot 2^{-\alpha m}$ for all $J\in\mathcal{D}'$ and some $\alpha\ge 0$.
\item[\textup{(2)}]  $\sum_{J\in\mathcal{D}'} \mu(J)^{q_0} \ge 2^{-(\tau(q_0)+\delta)m}$.
\end{enumerate}
 Then $|\mathcal{D}'|\le 2^{m(\tau^*(\alpha_0)+\e)}$.
\end{lemma}
\begin{proof}
Set $\eta:=\e/(3q_0)$, and pick $\delta\le \eta^2/9$, and also small enough that, if $q_1=q_0-\delta^{1/2}$, then
\begin{equation} \label{eq:using-tau-diff-1}
 \tau(q_0)-\tau(q_1) \le \delta^{1/2}\alpha_0 + \delta^{1/2}\eta .
\end{equation}
On one hand, using (1) and the definition of $\tau(q)$, we get
\begin{equation*}
2^{-(\tau(q_1)-\delta)m} \ge \|\mu^{(m)}\|_{q_1}^{q_1} \ge |\mathcal{D}'| 2^{- \alpha  q_1 m},
\end{equation*}
if $m$ is large enough (depending on $q_0, \tau$). On the other hand, by the assumptions (1)--(2),
\[
|\mathcal{D}'| 2^{-\alpha q_0 m} \ge 2^{-q_0} 2^{(-\tau(q_0)-\delta)m} \ge 2^{(-\tau(q_0)-2\delta)m}
\]
if $m\gg_{\delta,q_0} 1$. Eliminating $|\mathcal{D}'|$ from the last two displayed equations yields
\[
\alpha q_0 -\tau(q_0) -2\delta \le \alpha (q_0-\delta^{1/2})-\tau(q_0-\delta^{1/2})+\delta,
\]
so that, recalling \eqref{eq:using-tau-diff-1},
\[
\delta^{1/2}\alpha \le \tau(q_0)-\tau(q_0-\delta^{1/2}) +3\delta \le \delta^{1/2}\alpha_0 +\delta^{1/2}\eta+3\delta.
\]
Hence $\alpha-\alpha_0< 2\eta$, since we assumed $\delta\le (\eta/3)^2$.  Using this, we get that if $m\gg_\e 1$, then
\[
2^{(-\tau(q_0)+\e/3)m} \ge \|\mu^{(m)}\|_{q_0}^{q_0} \ge 2^{-q_0\alpha m} |\mathcal{D}'| \ge 2^{-q_0\alpha_0 m} 2^{-(q_0 2 \eta) m} |\mathcal{D}'|.
\]
The conclusion follows from the formula $\tau^*(\alpha_0)=q_0 \alpha_0-\tau(q_0)$ and our choice $\eta=\e/(3q_0)$.
\end{proof}

\begin{lemma}  \label{lem:Lq-sum-large-mass}

Let $q_0>0$ be such that $\alpha_0=\tau'(q_0)$ exists. Given $\sigma>0$, there is $\e=\e(\sigma,q_0)>0$ such that the following holds for large enough $m$ (in terms of $\sigma, q_0$):
\begin{equation} \label{eq:Lq-sum-large-mass}
\sum \{ \mu(J)^{q_0}: J\in\mathcal{D}_m, \mu(J) \ge 2^{-m(\alpha_0-\sigma)} \} \le 2^{-m(\tau(q_0)+\e)}.
\end{equation}
\end{lemma}
\begin{proof}
Let $\eta\in (0,1)$ be small enough that
\begin{equation}   \label{eq:using-tau-diff-2}
\tau(q_0+\eta) \ge \tau(q_0) + \eta \alpha_0 - \delta,
\end{equation}
where $\delta= \eta\sigma/(4+2 q_0)$.

Let $\alpha_j = \alpha_0 -\delta j$, and write $N(\alpha_j,m)$ for the number of intervals $J$ in $\mathcal{D}_m$ such that $2^{-m\alpha_j}\le\mu(J)< 2^{-m\alpha_{j+1}}$. For any fixed value of $q$, if $m\gg_q 1$ then,
\begin{equation*}
N(\alpha_j,m) 2^{-m q\alpha_j} \le \|\mu^{(m)}\|_q^q \le 2^{-m(\tau(q)-\delta)}.
\end{equation*}
Applying this to $q=q_0+\eta$, and using \eqref{eq:using-tau-diff-2}, we estimate
\begin{align*}
N(\alpha_j,m) 2^{-m q_0\alpha_j} &\le 2^{m \eta \alpha_j} 2^{-m(\tau(q_0+\eta)-\delta)} \\
&\le 2^{2\delta m} 2^{-j\delta\eta m} 2^{-\tau(q_0)m}.
\end{align*}
Let $\mathcal{S}$ be the sum in the left-hand side of \eqref{eq:Lq-sum-large-mass} that we want to estimate. Using that $\delta=\eta\sigma/(4+2 q_0)$, we conclude that
\begin{align*}
\mathcal{S}  &\le \sum_{j: \delta (j+1)\ge \sigma} N(\alpha_j,m) 2^{-m q_0\alpha_{j+1}} \\
&\le \sum_{j: \delta (j+1)\ge \sigma}  2^{\delta q_0 m}  2^{2\delta m} 2^{-j\delta\eta m} 2^{-\tau(q_0)m}\\
&\le \sum_{j\ge 0} 2^{-j\delta\eta m} 2^{(2+q_0)\delta m}  2^{-\eta\sigma m} 2^{-\tau(q_0)m} \\
&\le O_{\delta\eta}(1) 2^{(\eta\sigma/2-\eta\sigma)m} 2^{-\tau(q_0)m},
\end{align*}
as claimed.
\end{proof}

\begin{lemma} \label{lem:Lq-sum-over-small-set}
Let $q_0>1$ be such that $\alpha_0=\tau'(q_0)$ exists. Given $\kappa>0$, there is $\e=\e(\kappa,q_0)>0$ such that the following holds for large enough $m$ (in terms of $q_0,\e$).

If $\mathcal{D}'\subset\mathcal{D}_m$ has $\le 2^{(\tau^*(\alpha_0)-\kappa)m}$ elements, then
\[
\sum_{J\in\mathcal{D}'} \mu(J)^{q_0} \le 2^{-(\tau(q_0)+\e)m}.
\]
\end{lemma}
\begin{proof}
Let $\sigma=\kappa/(2 q_0)$. In light of Lemma \ref{lem:Lq-sum-large-mass}, we only need to worry about those $J$ with $\mu(J)\le 2^{-m(\alpha_0-\sigma)}$. But
\begin{align*}
\sum\{ \mu(J)^{q_0}: J\in\mathcal{D}',\mu(J) \le 2^{-m(\alpha_0-\sigma)} \} &\le 2^{(\tau^*(\alpha_0)-\kappa)m} 2^{-(q_0\alpha_0-q_0\sigma)m} \\
&=   2^{-(\kappa-q_0\sigma)m}2^{-\tau(q_0)m}.
\end{align*}
By our choice of $\sigma$, $\kappa-q_0\sigma=\kappa/2>0$, so this gives the claim.
\end{proof}

The results in this section so far hold for general measures. The following proposition, on the other hand, relies crucially on self-similarity. The second part was first proved in \cite{PeresSolomyak00}. Since the claim of Theorem \ref{thm:main-homogeneous} is not affected by rescaling and translating $\mu$ (from the point of view of the IFS, this amounts to doing these operations on the translation parameters $t_i$), from now on we assume that $\mu$ is supported on $[0,1)$.
\begin{prop}   \label{prop:Lq-over-small-set-is-small}
Let $q>1$ be such that $\alpha=\tau'(q)$ exists.
\begin{enumerate}[label={\upshape(\roman*).}]
\item Given $\kappa>0$, there is $\eta=\eta(\kappa,q)>0$ such that the following holds for all large enough $m$: for any $s\in\N$, $J\in\mathcal{D}_s$, if $\mathcal{D}'$ is a collection of intervals in $\mathcal{D}_{s+m}(J)$ with $|\mathcal{D}'|\le 2^{(\tau^*(\alpha)-\kappa)m}$, then
\[
\sum_{J\in\mathcal{D}'} \mu(J)^q  \le 2^{-(\tau(q)+\eta)m} \mu(2 I)^q.
\]
\item Given $\delta>0$, the following holds for all large enough $m$: for any $I\in\mathcal{D}_s$, $s\in\N$,
\[
\sum_{J\in\mathcal{D}_{s+m}(I)} \mu(J)^q \le 2^{-(\tau(q)-\delta)m} \mu(2 I)^q.
\]
\end{enumerate}
\end{prop}
\begin{proof}
We prove (i) first. Let $n$ be the smallest integer such that $\lam^n < 2^{-s-2}$. Let $y_j$ be the atoms of $\mu_{n}$ such that $[y_j,y_j+\lam^{n}]\cap I\neq \varnothing$, let $\wt{p}_j$ be their respective masses, and write
\[
\mu_{n,I} = \sum_j \wt{p}_j \delta_{y_j}.
\]
Then the support of $\mu_{n,I}$ is contained in the $\lam^n$-neighborhood of $I$. Moreover, since $\delta_z * S_{\lam^n}\mu$ is supported on $[z,z+\lam^n]$, as we assumed that $\mu$ is supported on $[0,1]$ , it follows from the self-similarity relation $\mu = \mu_{n}* S_{\lam^n}\mu$ and the definition of $\mu_{n,I}$ that  $\mu|_I =(\mu_{n,I}*S_{\lam^n}\mu)|_I$. Write
\[
\wt{p} = \|\mu_{n,I}\|_1 = \sum_j \wt{p}_j \le \mu(2 I),
\]
using that the support of $\mu_{n}$ is contained in the $\lam^n$-neighborhood of the support of $\mu$, and that $4 \lam^n \le 2^{-s}$.

We can then estimate
\begin{align*}
\sum_{J\in\mathcal{D}'}  \mu(J)^q &= \sum_{J\in\mathcal{D}'} \left( \sum_j \wt{p}_j \delta_{y_j} * S_{\lam^n}\mu(J)  \right)^q\\
&= \sum_{J\in\mathcal{D}'} \left( \sum_j \wt{p}_j \mu(\lam^{-n}(J-y_j)) \right)^q\\
&\le \sum_{J\in\mathcal{D}'} \wt{p}^{q-1} \sum_j \wt{p}_j \, \mu(\lam^{-n}(J-y_j))^q\\
&= \wt{p}^{q-1}\sum_j \wt{p}_j \sum_{J\in\mathcal{D}'} \mu(\lam^{-n}(J-y_j))^q,
\end{align*}
where we used the convexity of $t^q$ in the third line. Now for each fixed $j$, each interval $\lam^{-n}(J-y_j)$ with $J\in\mathcal{D}'$ can be covered by $O_{\lam}(1)$ intervals in $\mathcal{D}_m$, and reciprocally each interval in $\mathcal{D}_m$ hits at most $2$ intervals among the $\lam^{-n}(J-y_j)$. We deduce from Lemmas \ref{lem:Holder} and \ref{lem:Lq-sum-over-small-set} that, still for a fixed $j$,
\[
 \sum_{J\in\mathcal{D}'} \mu(\lam^{-n}(J-y_j))^q \le O_{\lam,q}(1) 2^{-(\tau(q)+\e)m},
\]
provided $m$ is taken large enough, where $\e=\e(\kappa,q)>0$ is given by Lemma \ref{lem:Lq-sum-over-small-set}. Combining the last three displayed equations yields the first claim with $\eta=\e/2$.

The second claim follows in the same way, adding over $\mathcal{D}_{s+m}(I)$ instead of $\mathcal{D}'$.
\end{proof}

\subsection{Proof of Theorem \ref{thm:conv-with-ssm-flattens}}

\label{subsec:conv-ssm-flattens}

In this section we prove Theorem \ref{thm:conv-with-ssm-flattens}. A similar result, with smoothness measured by entropy rather than $L^q$ norms, was proved by Hochman in \cite[Corollary 5.5]{Hochman14}, using his inverse theorem for the entropy of convolutions. In Hochman's approach, a crucial property of self-similar measures is that their entropy is roughly constant at most scales and locations, a property that Hochman termed \emph{uniform entropy dimension}, see \cite[Definition 5.1 and Proposition 5.2]{Hochman14} for precise details. Unfortunately,  there is no useful analog of the notion of uniform entropy dimension for $L^q$ norms. One of the key differences is that nearly all of the $L^q$ norm may be (and often is) captured by sets of extremely small measure; while sets of small measure also have small entropy. Instead, we will use the regularity of the multifractal spectrum established in the previous section in the following manner: if the flattening claimed in Theorem \ref{thm:conv-with-ssm-flattens} does not hold, then the inverse theorem provides a regular set $A$ which captures much of the $L^q$ norm of $\mu$. The upper bound on $\|\rho\|_q$, together with (5)--(6) in the inverse theorem imply that $A$ has nearly full branching for a positive proportion of $2^{D}$-scales, so it must have substantially less than average branching also on a positive proportion of scales. On the other hand, we will call upon the lemmas from the previous section to show that, in fact, $A$ must have nearly constant branching on nearly all scales (this is the part that uses the differentiability of $\tau$ at $q$), obtaining the desired contradiction.

\begin{proof}[Proof of Theorem \ref{thm:conv-with-ssm-flattens}]
Suppose $\rho$ is a $2^{-m}$-measure with $\|\rho\|_q^{q'} \le 2^{-\sigma m}$. In the course of the proof, we will choose many numbers which ultimately depend on $\sigma$ and $q$ only. To ensure that there is no circularity in their definitions, we indicate their dependencies: $\alpha=\alpha(q)$, $\kappa=\kappa(\alpha,\sigma)$, $\gamma=\gamma(q,\alpha,\kappa)$, $\delta'=\delta'(\alpha,\sigma,\kappa)$,   $\eta=\eta(q,\kappa)$, $\delta=\delta(q,\delta',\gamma,\eta)$, $\xi=\xi(q,\delta',\eta,\gamma)$,  $D_0=D_0(q,\sigma,\delta)$, $D=D(q,\delta,D_0)$, $\e=\e(q,\delta,D_0)$. Moreover, at different parts of the proof we will require $\delta',\delta,\xi$ to be smaller than certain (positive) functions of the parameters they depend on; in particular, all of the requirements can be satisfied simultaneously.

Finally, $m$ will be taken large enough in terms of all the previous parameters (hence ultimately in terms of $q$ and $\sigma$).

Write $\alpha=\tau'(q)$, and define $\kappa$ as
\begin{equation} \label{eq:def-kappa}
\kappa = (1-\tau^*(\alpha))\sigma/4.
\end{equation}
Then $\kappa>0$ thanks to Lemma \ref{lem:f-alpha-smaller-than-one}, and the assumption $\tau(q)<q-1$.

We apply Proposition \ref{prop:Lq-over-small-set-is-small} to obtain a sufficiently large $D_0$  (in terms of $\delta,\sigma,q$, with $\delta$ yet to be specified) such that
\begin{enumerate}
\item[(A)] For any $D'\ge D_0-2$, any $I\in\mathcal{D}_{s'}$, $s'\in\N$, and any subset $\mathcal{D}'\subset \mathcal{D}_{s'+D'}(J)$ with $|\mathcal{D}'|\le 2^{(\tau^*(\alpha)-\kappa)D'}$,
\[
\sum_{J'\in\mathcal{D}'} \mu(J')^q \le 2^{-(\tau(q)+\eta)D'}\mu(2J)^q,
\]
where $\eta$ depends on $\kappa$ and $q$, hence on $\sigma,q$ only.
\item[(B)] For any $D'\ge D_0-2$ and any $J\in\mathcal{D}_{s'}$, $s'\in\N$,
\[
\sum_{J'\in\mathcal{D}_{s'+D'}(J)} \mu(J')^q \le 2^{-(\tau(q)-\delta)D'}\mu(2J)^q.
\]
\item[(C)] $1/D_0 <\delta$.
\end{enumerate}

Let $\e>0,D\in\N$ be the numbers given by Theorem  \ref{thm:inverse-thm} applied to $\delta, D_0$ and $q$. For the sake of contradiction, suppose
\[
\| \rho*\mu^{(m)}\|_q \ge 2^{-\e m} \|\mu^{(m)}\|_q.
\]
We will derive a contradiction from this provided $m=\ell D$ is large enough (if $m$ is not of the form $\ell D$, we apply the argument to $\lfloor m/D\rfloor D$ instead; we omit the details). We apply Theorem \ref{thm:inverse-thm} to $\rho$ and $\mu^{(m)}$ to obtain (assuming $m$ is large enough) a set $A\subset\supp(\mu^{(m)})$ as in the theorem, with corresponding branching numbers $R'_s$. Since translating $\rho$ and $\mu^{(m)}$ does not affect their norms or the norm of their convolution, we assume for simplicity that the numbers $k_A, k_B$  are both $0$.

The key to the proof is to show, using the structure of $A$ provided by Theorem \ref{thm:inverse-thm}, that
\begin{equation} \label{eq:small-branching-at-many-scales}
|\{ s\in [\ell]: R'_s \le 2^{(\tau^*(\alpha)-\kappa)D} \}| \ge \gamma \ell,
\end{equation}
where $\gamma>0$ depends on $q,\alpha$ and $\kappa$ only (and $\kappa$ is given by \eqref{eq:def-kappa}). We first show how to complete the proof assuming this. Consider the sequence
\[
L_s = -\log\sum_{J\in\mathcal{D}_{s D}(A)} \mu(J)^q.
\]
By (B) applied with $s'=sD+2$ and $D'=D-2$,
\[
L_{s+1} \ge (\tau(q)-\delta)(D-2) -\log \sum_{J\in \mathcal{D}_{sD+2}(A)} \mu(2J)^q.
\]
But if $J\in \mathcal{D}_{sD+2}(A)$, then $2J$ is contained in a single interval in $\mathcal{D}_{sD}(A)$ by property \ref{A4} from Theorem \ref{thm:inverse-thm}, and conversely $J'\in\mathcal{D}_{sD}(A)$ hits at most two intervals $2J$, $J\in\mathcal{D}_{sD+2}(A)$. We deduce that
\[
L_{s+1} \ge L_s +(\tau(q)-\delta)(D-2)-1
\]
for all $s\in[\ell]$. Likewise, by (A),
\[
L_{s+1} \ge L_s + (\tau(q)+\eta)(D-2)-1,
\]
whenever $R'_s \le  2^{(\tau^*(\alpha)-\kappa)D}$. Recall that $\eta$ depends on $q,\kappa$. In light of \eqref{eq:small-branching-at-many-scales}, and using also (C), we have
\begin{align*}
L_{\ell} &\ge (\tau(q)+\eta)\gamma \ell (D-2) + (\tau(q)-\delta)(1-\gamma)\ell (D-2)-\ell\\
 &\ge (\tau(q)+\eta\gamma-\delta(1-\gamma))m - 2\delta (\tau(q)+\eta) m -\delta m.
\end{align*}

Hence, by choosing $\delta$ small enough in terms of $\tau(q),\gamma$ and $\eta$ we can ensure that, for $m$ large enough,
\[
L_{\ell} = -\log\|\mu^{(m)}|_A\|_q^q \ge (\tau(q)+\eta\gamma/2)m.
\]
On the other hand, by \ref{A1} in Theorem \ref{thm:inverse-thm}, if $\xi>0$ is a small number to be fixed later, then (always assuming $m$ is large enough)
\[
\|\mu^{(m)}|_A\|_q^q \ge 2^{-q\delta m}\|\mu^{(m)}\|_q^q \ge 2^{-q\delta m} 2^{-(\tau(q)+\xi)m}.
\]
From the last two displayed equations,
\[
\eta\gamma/2 \le q\delta +\xi.
\]
Recall that $\eta=\eta(\kappa,q),\gamma=\gamma(q,\alpha,\kappa)$ is yet to be specified, while $\delta$ so far was taken small enough in terms of $\tau(q),\gamma$ and $\eta$, and no conditions have been yet imposed on $\xi$. By ensuring $q\delta <\eta\gamma/8$ and $\xi\le \eta\gamma/8$ we reach a contradiction, as desired.

It remains to establish \eqref{eq:small-branching-at-many-scales}. The idea is very simple: Theorem \ref{thm:inverse-thm} (together with the assumption that $\|\rho\|_q^{q'}\le 2^{-\sigma m}$) imply that $A$ has ``nearly full branching'' on a positive proportion of scales. On the other hand, Lemma \ref{lem:size-set-A-in-terms-of-f-alpha} says the size of $A$ is at most roughly $2^{\tau^*(\alpha)m}\ll 2^m$ (by Lemma \ref{lem:f-alpha-smaller-than-one}), so there must be a positive proportion of scales on which the average $2^D$-adic branching is far smaller than $2^{\tau^*(\alpha)D}$, which is what \eqref{eq:small-branching-at-many-scales} says.

We proceed to the details. Using \ref{A1}, \ref{A2} in Theorem \ref{thm:inverse-thm}, we get that (for $m\gg_\delta 1$) there is $\wt{\alpha}>0$ such that $\mu^{(m)}(a)\in [2^{-\wt{\alpha} m},2^{1-\wt{\alpha} m}]$ for all $a\in A$, and
\begin{equation*} 
\sum_{J\in\mathcal{D}_m(A)} \mu(J)^q \ge 2^{-q\delta m} \sum_{J\in\mathcal{D}_m}\mu(J)^q \ge  2^{-(\tau(q)+ q\delta+\xi)m}.
\end{equation*}
We let $\delta\le \delta'$  and $\xi$ be small enough in terms of $\delta'$ and $q$ that, invoking Lemma \ref{lem:size-set-A-in-terms-of-f-alpha},
\begin{equation} \label{eq:size-A-upper-bound}
|A| \le 2^{(\tau^*(\alpha)+\delta')m}.
\end{equation}

Let $\mathcal{S}'=[\ell]\setminus \mathcal{S}$, where $\mathcal{S} = \{ s: R'_s \ge 2^{(1-\delta)D}\}$. Using \ref{A3} in Theorem \ref{thm:inverse-thm}, we see that
\begin{equation} \label{eq:size-A-lower-bound}
|A| = \prod_{s=0}^{\ell-1} R'_s \ge  2^{(1-\delta)D|\mathcal{S}|} \prod_{s\in\mathcal{S}'} R'_s.
\end{equation}
Let $m_1=D|\mathcal{S}|$, $m_2=D|\mathcal{S}'|=m-m_1$. Combining \eqref{eq:size-A-upper-bound} and \eqref{eq:size-A-lower-bound}, and using that $\delta\le \delta'$, we deduce
\begin{equation} \label{eq:upper-bound-Rs-S-prime}
\prod_{s\in\mathcal{S}'} R'_s \le 2^{-(1-\delta)m_1} 2^{(\tau^*(\alpha)+\delta')m} \le 2^{-(1-\tau^*(\alpha)-2\delta')m_1} 2^{(\tau^*(\alpha)+\delta')m_2}.
\end{equation}
Note that $1-\tau^*(\alpha)>0$ by Lemma \ref{lem:f-alpha-smaller-than-one}. At this point we take $\delta'$ small enough that $1-\tau^*(\alpha)-2\delta'>0$.
Using \ref{inverse6} in Theorem \ref{thm:inverse-thm}, and the assumption $\|\rho\|_q^{q'}\le 2^{-\sigma m}$, we further estimate
\begin{equation} \label{eq:bound-total-branching-scales}
(\sigma - \delta)m  \le m_1 \le \left( (\tau(q)+\xi)/(q-1)+\delta\right)m.
\end{equation}
We can plug in the left inequality (together with $m_2\le m$) into \eqref{eq:upper-bound-Rs-S-prime}, to obtain the key estimate
\[
\log\prod_{s\in\mathcal{S}'} R'_s \le \left(\tau^*(\alpha)+ \delta' -(1-\tau^*(\alpha)-2\delta')(\sigma-\delta)\right)m_2.
\]
Recalling \eqref{eq:def-kappa}, this shows that by making $\delta'$ (hence also $\delta\le\delta'$) small enough in terms of $\alpha,\sigma,\kappa$, we have
\[
\log\prod_{s\in\mathcal{S}'} R'_s \le (\tau^*(\alpha)-2\kappa) m_2.
\]
Let $\mathcal{S}_1 = \{ s\in\mathcal{S}': \log R'_s \le (\tau^*(\alpha)-\kappa)D\}$. Recall that our goal is to show \eqref{eq:small-branching-at-many-scales}, i.e. $|\mathcal{S}_1|\ge \gamma(q,\alpha,\kappa)\ell$. We have
\[
D|\mathcal{S}'\setminus\mathcal{S}_1| \le \frac{1}{\tau^*(\alpha)-\kappa}\sum_{s\in\mathcal{S}'\setminus\mathcal{S}_1} \log R'_s \le \frac{\tau^*(\alpha)-2\kappa}{\tau^*(\alpha)-\kappa} D|\mathcal{S}'|,
\]
so that, using the right-most inequality in \eqref{eq:bound-total-branching-scales}, and recalling that $D|\mathcal{S}'|=m-m_1$,
\[
D|\mathcal{S}_1| \ge \frac{\kappa (m-m_1)}{\tau^*(\alpha)-\kappa} \ge \left(\frac{\kappa(1-(\tau(q)+\xi)/(q-1)-\delta)}{\tau^*(\alpha)-\kappa}\right)  m.
\]
By ensuring that $\delta,\xi$ are small enough in terms of $q$, the right-hand side above can be bounded below by
\[
\left(\frac{\kappa(1-\tau(q)/(q-1))/2}{\tau^*(\alpha)-\kappa}\right) m,
\]
confirming that \eqref{eq:small-branching-at-many-scales} holds with $\gamma=\gamma(q,\alpha,\kappa)$.
\end{proof}

\subsection{Proof of Theorem \ref{thm:main-homogeneous}}

Theorem \ref{thm:main-homogeneous} will be an easy consequence of the following proposition,  which relies on Theorem \ref{thm:conv-with-ssm-flattens}. It is an analog of \cite[Theorem 1.4]{Hochman14}, and we follow a similar outline. We emphasize that exponential separation is not required for the validity of the proposition.
\begin{prop} \label{prop:ssm-scale-Rm-norm}
Let $q\in (1,\infty)$ be such that $\tau$ is differentiable at $q$ and $\tau(q)<q-1$. Fix $R\in\N$. Then
\[
\lim_{n\to\infty} \frac{\log \|\mu_{n}^{(Rm(n))}\|_q^q}{n\log\lam} = \tau(q),
\]
where $m(n)$ is the smallest integer with $2^{-m(n)}\le \lam^n$.
\end{prop}
\begin{proof}
Fix $n\in\N$. We write $m=m(n)$ for simplicity, and allow all implicit constants to depend on $q$  only. Using the self-similarity relation $\mu=\mu_n * S_{\lambda^n}\mu$ and Lemma \ref{lem:discr-norm-conv-equivalence}, we get
\begin{align*}
\|\mu^{((R+1)m)}\|_q^q &\le O(1) \| \mu_{n}^{((R+1)m)} * (S_{\lam^n} \mu)^{((R+1)m)} \|_q^q \\
&= O(1) \big\|  \sum_{J\in\mathcal{D}_m} \mu_{n}(J) (\mu_{n})_J^{((R+1)m)} * (S_{\lam^n} \mu)^{((R+1)m)} \big\|_q^q.
\end{align*}
Here $(\mu_{n})_J  = \mu_{n}|_J / \mu_{n}(J)$ is the normalized restriction of $\mu_{n}$ to $J$ (note that we are only summing over $J$ such that $\mu_{n}(J)>0$). Since the measures $(\mu_{n})_J^{((R+1)m)} * (S_{\lam^n} \mu)^{((R+1)m)}$ are supported on $J+[0,\lam^n]$, the support of each of them hits the supports of $O(1)$ others. We can then apply Lemma \ref{lem:L-q-norm-almost-disj-supports} to obtain
\[
\|\mu^{((R+1)m)}\|_q^q  \le O(1)\sum_{J\in\mathcal{D}_m} \mu_{n}(J)^q  \|(\mu_{n})_J^{((R+1)m)} * (S_{\lam^n} \mu)^{((R+1)m)} \|_q^q
\]
Let $\rho_{J} = S_{\lam^{-n}}(\mu_{n})_J$ (we suppress the dependence on $n$ from the notation, but keep it in mind). Note that $S_a(\eta)*S_a(\eta')=S_a(\eta*\eta')$ for any $a>0$ and measures $\eta,\eta'$. It follows from Lemmas \ref{lem:Holder} and \ref{lem:discr-norm-conv-equivalence}  that
\[
 \|(\mu_{n})_J^{((R+1)m)} * (S_{\lam^n} \mu)^{((R+1)m)} \|_q^q \le O(1) \| \rho_{J}^{(Rm)} * \mu^{(Rm)} \|_q^q,
\]
so that, combining the last two displayed formulas,
\begin{equation} \label{eq:estimate-mu-Rp1}
\|\mu^{((R+1)m)}\|_q^q  \le O(1) \sum_{J\in\mathcal{D}_m} \mu_{n}(J)^q \| \rho_{J}^{(Rm)} * \mu^{(Rm)} \|_q^q.
\end{equation}
On the other hand,  using Lemma \ref{lem:Holder} again,
\begin{equation} \label{eq:estimate-mu-m-Rp1}
\|\mu_{n}^{((R+1)m)}\|_q^q = \sum_{J\in\mathcal{D}_m} \mu_{n}(J)^q \|(\mu_{n})_J^{((R+1)m)}\|_q^q \ge  \Omega(1) \sum_{J\in\mathcal{D}_m} \mu_{n}(J)^q \| \rho_{J}^{(Rm)} \|_q^q.
\end{equation}

Fix $\sigma>0$, and let $\mathcal{D}' = \{ J\in\mathcal{D}_{m}: \| \rho_{J}^{(Rm)}\|_q^q \le 2^{-\sigma m}\}$. According to Theorem \ref{thm:conv-with-ssm-flattens}, there is $\e=\e(\sigma,q)>0$ such that, if $n$ is taken large enough, then
\[
J\in\mathcal{D}' \quad\Longrightarrow\quad\|\rho_{J}^{(Rm)} * \mu^{(Rm)}\|_q^q \le 2^{-(\tau(q)+\e)Rm}.
\]
Applying this to \eqref{eq:estimate-mu-Rp1},  we get
\begin{align*}
\|\mu^{((R+1)m)}\|_q^q  &\le O(1) 2^{-(\tau(q)+\e)Rm} \sum_{J\in\mathcal{D}'} \mu_{n}(J)^q  + O(1) \sum_{J\notin \mathcal{D}'}  \mu_{n}(J)^q \|  \mu^{(Rm)} \|_q^q \\
&\le O(1) 2^{-(\tau(q)+\e)Rm} \|\mu^{(m)}\|_q^q  + O(1)\|  \mu^{(Rm)} \|_q^q \sum_{J\notin\mathcal{D}'} \mu_{n}(J)^q\\
\end{align*}
using Young's inequality in the first line, and Lemma \ref{lem:comparison-mu-m-mu} in the second. On the other hand,
\[
 2^{-(\tau(q)+\e)Rm} \|\mu^{(m)}\|_q^q \le 2^{-\e m/2} \|\mu^{((R+1)m)}\|_q^q
\]
if $n$ is large enough (depending on $R$). Inspecting the last two displayed equations, we deduce that if $n\gg_{\sigma} 1$, then
\[
\sum_{J\notin \mathcal{D}'}  \mu_{n}(J)^q \ge \Omega(1) \frac{\|\mu^{((R+1)m)}\|_q^q}{\|\mu^{(Rm)} \|_q^q } \ge 2^{-m(\tau(q)+\sigma)}.
\]
 Recalling \eqref{eq:estimate-mu-m-Rp1}, we conclude that
\begin{align*}
\|\mu_{n}^{((R+1)m)}\|_q^q  &\ge \Omega(1) \sum_{J\notin\mathcal{D}'} \mu_{n}(J)^q \| \rho_{J}^{(Rm)} \|_q^q \\
&\ge \Omega(1) 2^{-\sigma m}\sum_{J\notin\mathcal{D}'} \mu_{n}(J)^q \ge \Omega(1) 2^{-2\sigma m} 2^{-m \tau(q)}.
\end{align*}
The inequality $\|\mu_{n}^{((R+1)m)}\|_q^q \le \|\mu_{n}^{(m)}\|_q^q$ holds trivially, so that by Lemma \ref{lem:comparison-mu-m-mu}
\[
\|\mu_{n}^{((R+1)m)}\|_q^q  \le \|\mu_{n}^{(m)}\|_q^q \le 2^{\sigma m} 2^{-m \tau(q)},
\]
provided $n\gg_\sigma 1$. Since $\sigma>0$ was arbitrary and $2^{-m}=\Theta(\lam^n)$, this concludes the proof.
\end{proof}

We can now conclude the proof of Theorem \ref{thm:main-homogeneous}
\begin{proof}[Proof of Theorem \ref{thm:main-homogeneous}]
We continue to write $m=m(n)=\lceil n\log(1/\lam)\rceil$. To begin, we note that,  for any $q\in (1,\infty)$,
\begin{equation} \label{eq:main-thm-upper-bound-1}
\|\mu_{n}^{(m)}\|_q^q \ge \|\mu_{n}\|_q^q \ge  \|\Delta\|_q^{qn}.
\end{equation}
(The latter inequality is an equality if and only if there are no overlaps among the atoms of $\mu_{n}$.)  Since $\|\nu^{(m)}\|_q^{q'} \ge 2^{-m}$ for any probability measure $\nu$, it follows from \eqref{eq:main-thm-upper-bound-1} and Lemma \ref{lem:comparison-mu-m-mu} that
\[
D(q) \le \min(\sdim(\mu,q),1).
\]
Hence the proof will be completed if we can show that for each $q\in (1,\infty)$, either $\tau(q)\ge q-1$ (so that in fact $\tau(q)=q-1$) or
\begin{equation} \label{eq:formula-tau-q-want-to-prove}
\tau(q)= \log\|\Delta\|_q^q.
\end{equation}
Since $\tau(q)$ is concave, it is enough to prove this for all $q$ such that $\tau$ is differentiable at $q$. Hence, we fix $q$ such that $\tau(q)<q-1$ and $\tau$ is differentiable at $q$, and we set out to prove \eqref{eq:formula-tau-q-want-to-prove}.

By the exponential separation assumption, the atoms of $\mu_{n}$ are $\lam^{R n}$-separated for infinitely many $n$ and some $R\in\N$. We know from Proposition \ref{prop:ssm-scale-Rm-norm}  that
\begin{equation} \label{eq:tau-mu-x-equals-T}
\lim_{n\to\infty} \frac{\log \|\mu_{n}^{(Rm(n))}\|_q^q}{n\log\lam} = \tau(q).
\end{equation}
On the other hand, if $n$ is such that the atoms of $\mu_{n}$ are $\lam^{R n}$-separated then (since $\lam^{Rn}\ge 2^{-R m(n)}$)
\begin{equation} \label{eq:tau-mu-x-exp-separation}
\|\mu_{n}^{(Rm(n))}\|_q^q = \|\mu_{n}\|_q^q =  \| \Delta\|_q^{qn}.
\end{equation}
 Combining Equations \eqref{eq:tau-mu-x-equals-T} and \eqref{eq:tau-mu-x-exp-separation}, we conclude that \eqref{eq:formula-tau-q-want-to-prove} holds, finishing the proof.
\end{proof}

\subsection{About the proof of Theorem \ref{thm:main}}

In the proof of Theorem \ref{thm:main-homogeneous} the convolution structure played a crucial role. While a general self-similar measure does not have such a clean convolution structure, we can proceed as follows. Let $(\lam_i)_{i\in I}$ be the scaling factors of the IFS generating $\mu$ (there may be repetitions). Given $m$, let
\begin{align*}
\Omega_m &= \{ (j_1\ldots j_k): \lambda_{j_1}\cdots \lambda_{j_k} \le 2^{-m} < \lambda_{j_1}\cdots \lambda_{j_{k-1}}\},\\
\Lambda_m &= \{ \lambda_{j_1}\cdots \lambda_{j_k}: (j_1\ldots j_k)\in\Omega_m\}.
\end{align*}
One can then check, using self-similarity, that
\[
\mu = \sum_{\lam\in\Lambda_m} \mu_{\lam,m} * S_{\lam}\mu,
\]
where $\mu_{\lam,m}$ are certain purely atomic measures constructed from the translations of the maps $f_{j_1}\cdots f_{j_k}$ with $\lam_{j_1}\cdots \lam_{j_k}=\lam$. Thanks to the fact that $|\Lambda_m|$ is polynomial in $m$ (even though $|\Omega_m|$ is exponential in $m$), the proof given in the homogeneous case can be adapted with minor technical complications. We refer to \cite[\S 6.4]{Shmerkin19} for the details.

\section{Applications}

\label{sec:applications}

In this section we present several applications of Theorem \ref{thm:main-homogeneous}.

\subsection{Frostman exponents}
\label{subsect:appl-Frostman}

If $\mu$ is a finite measure on a metric space $X$, we say that $\mu$ has \emph{Frostman exponent} $s$ if $\mu(B(x,r)) \le C\,r^s$ for some $C>0$ and all $x\in X,r>0$. There is a very simple relation between $L^q$ dimensions for large $q$ and Frostman exponents:
\begin{lemma} \label{lem:Lq-dim-to-Frostman-exp}
Let $\mu\in\mathcal{P}$. If $D_\mu(q)> s$ for some $q\in (1,\infty)$, then there is $r_0>0$ such that
\[
\mu(B(x,r)) \le \, r^{(1-1/q)s} \text{ for all } x\in\R, r\in (0,r_0].
\]
\end{lemma}
\begin{proof}
If $D(\mu,q)>s$, then there is $s'>s$ such that for all large enough $m$ and each $J'\in\mathcal{D}_m$,
\[
\mu(J')^q \le \sum_{J\in\mathcal{D}_m} \mu(J)^q \le 2^{-m (q-1) s'}.
\]
Since any ball can be covered by $O(1)$ dyadic intervals of size smaller than the radius, we get that if $r$ is sufficiently small then
\[
\mu(B(x,r)) \le C\, r^{(1-1/q)s'},
\]
where $C$ is independent of $x$ and $r$. This gives the claim.
\end{proof}

Theorem \ref{thm:main-homogeneous} together with the previous lemma immediately yields the following corollary:
\begin{corollary} \label{cor:Frostman-exp}
Let $\mu$ be the self-similar measure associated to a homogeneous IFS $(\lambda x+t_i)_{i\in I}$ with exponential separation and the uniform probability weights $(1/|I|,\ldots,1/|I|)$. Then $\mu$ has Frostman exponent $s$ for every $s< \min(\log |I|/\log(1/\lambda),1)$.
\end{corollary}

\subsection{Algebraic parameters}
\label{subsec:appl-algebraic}

We now discuss the special case in which the IFS has algebraic parameters, that is, both the contraction ratio $\lambda$ and the translations $t_i$ are algebraic numbers. Hochman \cite[Corollary 1.5]{Hochman14} proved that the overlaps conjecture holds in this case and, in the same way, we extend this to the $L^q$-dimension version of the overlaps conjecture. The deduction is based on the following classical lemma; see \cite[Lemma 6.30]{Hochman17} for a proof.

\begin{lemma} \label{lem:algebraic}
Given algebraic numbers (over $\mathbb{Q}$) $\alpha_1,\ldots,\alpha_k$ and a positive integer $h$, there exists $\delta>0$ such that the following holds: if $P\in\Z[x_1,\ldots,x_k]$ is a polynomial of degree $n$, all of whose coefficients are at most $h$ in modulus, then either $P(\alpha_1,\ldots,\alpha_k)=0$ or
\[
|P(\alpha_1,\ldots,\alpha_k)| \ge \delta^n.
\]
\end{lemma}

\begin{corollary} \label{cor:algebraic-no-overlaps}
Let $\mu$ be the self-similar measure associated to a homogeneous WIFS with algebraic coefficients (i.e. the contraction ratio and the translations are algebraic). Then either there is an exact overlap, or
\[
D_q(\mu) = \min(\sdim(\mu,q),1) \quad\text{for all }q>1.
\]
\end{corollary}
\begin{proof}
Note that for any pair of sequences $i=(i_1,\ldots,i_n)$, $j=(j_1,\ldots,j_n)$, the difference $f_i(0)-f_j(0)$ can be written as $P_{i,j}(\lambda,t_1,\ldots,t_{|I|})$, where $P_{i,j}\in\Z[x_1,\ldots,x_{|I|+1}]$ has degree at most $n+1$ and coefficients $\pm 1$. Since we assume that there are no exact overlaps, $P_{i,j}(\lambda,t_1,\ldots,t_{|I|})\neq 0$ for $i\neq j$.  Lemma \ref{lem:algebraic} then guarantees that the IFS has exponential separation, so that Theorem \ref{thm:main-homogeneous} yields the corollary.
\end{proof}

Even if there are exact overlaps, the proof of Theorem \ref{thm:main-homogeneous} yields an expression for the $L^q$ dimensions of $\mu$. Recall that $\mu_n$ is the purely atomic measure given by
\[
\mu_n = *_{j=0}^{n-1} S_{\lambda^j}\Delta = \sum_{u\in I^n}  p_{u_1}\cdots p_{u_n}\, \delta_{f_u(0)}.
\]
\begin{corollary}  \label{cor:algebraic}
Let $\mu$ be the self-similar measure associated to a homogeneous WIFS with algebraic coefficients (i.e. the contraction ratio and the translations are algebraic). Define
\[
T_\mu = \lim_{n\to\infty} -\frac{1}{n}\log \|\mu_n\|_q^q .
\]
Then the limit in this definition exists, and
\[
D_q(\mu) = \min\left(\frac{T_\mu}{(q-1)\log(1/\lambda)},1\right).
\]
\end{corollary}
\begin{proof}
By Lemma \ref{lem:algebraic}, and arguing as in the proof of Corollary \ref{cor:algebraic-no-overlaps}, there is $R\in\N$ such that any two distinct atoms of $\mu_n$ are $\lambda^{Rn}$-separated. Suppose $D_q(\mu)<1$. By Proposition \ref{prop:ssm-scale-Rm-norm},
\[
\lim_{n\to\infty} \frac{\log \|\mu_{n}^{(Rm(n))}\|_q^q}{n\log\lam} = \tau(q).
\]
But $\| \mu_n^{(Rm(n))}\|_q^q = \|\mu_n\|_q^q$ since $2^{m(n)}\le \lambda^n$, so the claim follows.
\end{proof}

\subsection{Parametrized families and absolute continuity}
\label{subsec:appl-abscont}

Exponential separation holds outside of a small set of exceptions in parametrized families satisfying mild regularity and non-degeneracy assumptions:
\begin{lemma} \label{lem:parametrized-family}
Let $J\subset\R$ be a compact interval, and let $\lambda:J\to (-1,0)\cup(0,1)$ and $t_1,\ldots,t_\ell:J\to\R$ be real-analytic functions. For a pair of $\{1,\ldots,\ell\}$-valued sequences $i,j$, define
\[
g_{i,j}(u) = \sum_{k=0}^\infty \lambda(u)^k t_{i_k}(u) - \sum_{k=0}^{n-1} \lambda(u)^k t_{j_k}(u).
\]
Assume that if $i\neq j$ then $g_{i,j}$ is not identically zero. Then the IFS $\{ \lambda(u)x+ t_i(u)\}_{i=1}^\ell$ has exponential separation for all $u$ outside of a set $E\subset J$ of zero Hausdorff (and even packing) dimension.
\end{lemma}
See \cite[Theorem 1.8]{Hochman14} for the proof and some further discussion. Now Theorem \ref{thm:main-homogeneous} shows that for parametrized families of WIFS satisfying the assumptions of Lemma \ref{lem:parametrized-family}, there is a zero-dimensional exceptional set of parameters outside of which the $L^q$ dimensions of the self-similar measures have the value predicted by the overlaps conjecture (note also that the exceptional set is independent of the probability weights).

When $\sdim(\mu,q)>1$, the overlaps conjecture predicts that (in the absence of exact overlaps) $D_\mu(q)=1$, but in fact it is plausible that under the same assumptions the measure $\mu$ is absolutely continuous with a $L^q$ density. While this is of course still open, and appears to be even harder than the dimension version of the overlaps conjecture, we have the following result for parametrized families:

\begin{theorem} \label{thm:abs-continuity}
Let $J\subset\R$ be an closed interval, and let $\lambda:J\to (-1,0)\cup(0,1)$ and $t_1,\ldots,t_\ell:J\to\R$ be real-analytic functions. For a pair of $\{1,\ldots,\ell\}$-valued sequences $i,j$, define
\[
g_{i,j}(u) = \sum_{k=0}^\infty \lambda(u)^k t_{i_k}(u) - \sum_{k=0}^{\infty} \lambda(u)^k t_{j_k}(u).
\]
Assume that if $i\neq j$ then $g_{i,j}$ is not identically zero. Then then there is a set $E\subset J$ of zero Hausdorff dimension such that the following holds for all $u\in J\setminus E$: if $\mu$ is a self-similar measure associated to the IFS $( \lambda(u)x+t_i(u))_{i=1}^\ell$ and a probability vector $(p_i)_{i=1}^\ell$, and if $\sdim(\mu,q)>1$ for some $q\in (1,\infty)$, then  $\mu$ is absolutely continuous and its Radon-Nikodym density is in $L^q$.
\end{theorem}
This theorem provides the correct range for the possibility of having an $L^q$ density (up to the endpoint), since measures $\mu$ with an $L^q$ density satisfy $D(\mu,q)=1$; this follows from the inequality $(\int_J f)^q \le |J|^{q-1} \int_J f^q$ for all intervals $J$, where $f$ is the $L^q$ density of $\mu$. The proof of the theorem follows the ideas from \cite{Shmerkin14,ShmerkinSolomyak16}; the only new element is the stronger input provided by Theorem \ref{thm:main-homogeneous}.

Recall that the Fourier transform of a  measure $\rho\in\mathcal{P}$  is defined as
\[
\widehat{\rho}(\xi) = \int \exp(2\pi i x\xi)\,d\rho(x).
\]
The following result asserts that convolving a measure of full $L^q$ dimension and another measure with power Fourier decay results in an absolutely continuous measure with an $L^q$ density; see \cite[Theorem 4.4]{ShmerkinSolomyak16} for the proof.
\begin{theorem} \label{thm:convolution-abs-cont}
Let $\nu,\rho\in\mathcal{P}$ be such that $D_\nu(q)=1$ for some $q>1$ and $\rho$ satisfies the Fourier decay estimate
\[
|\widehat{\rho}(\xi)| \le C|\xi|^{-\delta}
\]
for some $C,\delta>0$. Then the convolution $\nu*\rho$ is absolutely continuous and its Radon-Nikodym density is in $L^q$.
\end{theorem}
The proof of this theorem shows that, additionally, $\nu*\rho$ has fractional derivatives in $L^q$.

\begin{proof}[Proof of Theorem \ref{thm:abs-continuity}]
Fix a weight $(p_1,\ldots,p_\ell)$ for some $\ell\ge 2$. For $u\in J$, let $\mu_u$ be the self-similar measure associated with the WIFS $(\lambda(u)x+t_i(u),p_i)_{i=1}^\ell$. We also denote
\[
\Delta(u) = \sum_{i=1}^\ell p_i \, \delta_{t_i(u)}.
\]
Fix $k\in\N$. Using the convolution structure of $\mu_u$, we decompose
\begin{equation} \label{eq:decomp-convolution}
\mu_u = \left(*_{k\nmid j} S_{\lambda(u)^j}\Delta(u)\right)*\left(*_{k\mid j} S_{\lambda(u)^j}\Delta(u)\right)=:\nu_u^{(k)}*\rho_u^{(k)}.
\end{equation}
We can think of $\rho_u^{(k)}$ and $\nu_u^{(k)}$ as the measures obtained from the construction of $\mu_u$ by ``keeping only every $k$-th digit'' and ``skipping every $k$-th digit'' respectively. Both $\rho_u^{(k)}$ and $\nu_u^{(k)}$ are, again, self-similar measures arising from homogeneous IFS's. Indeed, $\rho_u^{(k)}$ is the invariant measure for the IFS $(\lambda(u)^k x+t_i(u),p_i)_{i=1}^\ell$. The WIFS generating $\nu^{(k)}$ is more cumbersome to write down: it consists of $\ell^{k-1}$ maps, indexed by sequences $i\in\{1,\ldots,\ell\}^{k-1}$. The maps and weights are given by
\begin{align*}
g_{u,i}(x) &= \lambda(u)^k(x)+ \sum_{j=0}^{k-2} t_{i_{j+1}} \lambda(u)^j,\\
p_i &= p_{i_1}\cdots p_{i_{k-1}}.
\end{align*}
A short calculation shows that, for any $q>1$,
\begin{equation} \label{eq:dim-nu-k}
\sdim(\nu_u^{(k)},q) = (1-1/k) \sdim(\mu_u,q).
\end{equation}
On the other hand, it is easy to check that (for each $k$) the family of IFS's generating $\nu_u^{(k)}$ also satisfies the assumptions of Lemma \ref{lem:parametrized-family}. Hence there are sets $E'_k$ of zero Hausdorff dimension such that the WIFS generating $\nu_u^{(k)}$ has exponential separation for all $u\in J\setminus E'_k$. Letting $E'=\cup_k E'_k$ and apply Theorem \ref{thm:main-homogeneous}, we deduce that $E'$ has zero Hausdorff dimension, and if $u\in J\setminus E$ then
\[
D(\nu_u^{(k)},q) = \min((1-1/k) \sdim(\mu_u),q) \quad\text{for all }k\in\N.
\]

Turning to the measures $\rho_u^{(k)}$, we claim hat there are exceptional sets $E''_k$ of zero Hausdorff dimension such that if $u\in J\setminus E''_k$, then $\rho_u^{(k)}$ has power Fourier decay, that is, there are $C(u,k),\delta(u,k)>0$ such that
\[
|\widehat{\rho_u^{(k)}}(\xi)| \le C(u,k)|\xi|^{-\delta(u,k)}.
\]
This follows by variants of an argument that goes back to Erd\H{o}s \cite{Erdos40}. If the function $\lambda(u)$ is nonconstant then, by splitting $J$ into finitely many intervals and reparametrizing, we may assume that $\lambda(u)=u$. This case is closer to Erd\H{o}s original argument; see e.g. \cite[Proposition 2.3]{Shmerkin14} for a detailed exposition. Suppose now that $\lambda(u)\equiv \lambda$. In this case we must have $\ell\ge 3$. Indeed, suppose $\ell=2$. Replacing $t_1(u)$ by $0$ and $t_2(u)$ by $1$ has the effect of rescaling and translating the measures $\mu_u$, which does not affect the claim. If $|\lam|<1/2$, then $\sdim(\mu,q)<1$ for any $q$ and there is nothing to do, while if $|\lam|\ge 1/2$, there are two sequences $i,j\in\{0,1\}^\N$ such that $\sum_{k=0}^\infty (i_k-j_k)\lam^k=0$, and this implies that the non-degeneracy assumption fails. Hence we assume that $\ell\ge 3$ from now on. In this case, the function
\[
h(u)=\frac{t_3(u)-t_1(u)}{t_2(u)-t_1(u)}
\]
is non-constant and real-analytic outside of a finite set of $u\in J$ (where the denominator vanishes). Otherwise, if either the denominator or $h(u)$ itself were constant, the non-degeneracy condition would fail. As before, this shows that we may assume $h(u)=u$. The claim now follows from \cite[Proposition 3.1]{ShmerkinSolomyak16}. Let $E''=\cup_k E''_k$.

Set $E=E'\cup E''$ and fix $u\in J\setminus E$. If $\sdim(\mu,q)>1$, then \eqref{eq:dim-nu-k} ensures that $D(\nu_u^{(k)},q)=1$ provided $k$ is taken large enough. Since also $\rho_u^{(k)}$ has power Fourier decay by the definition of $E''\subset E$, the decomposition \eqref{eq:decomp-convolution} together with Theorem \ref{thm:convolution-abs-cont} show that $\mu_u$ is absolutely continuous with an $L^q$ density, finishing the proof.
\end{proof}

\subsection{Bernoulli convolutions}
\label{subsect:appl-BCs}

Given $\lam\in (0,1)$, we define $\mu_\lambda$ as the distribution of the random sum $\sum_{n=0}^\infty X_n \lambda^n$, where the $X_n$ are IID and take values $0$ and $1$ with equal probability $1/2$. In other words, $\mu_\lambda$ is the self-similar measure associated to the WIFS $(\lambda x,1/2), (\lambda_x +1,1/2)$. The measures $\mu_\lambda$ are known as \emph{Bernoulli convolutions}.

When $\lam\in (0,1/2)$, the topological support of $\mu_\lambda$ is a self-similar Cantor set of dimension $\log 2/\log(1/\lambda)<1$; in particular, $\mu_\lambda$ is purely singular (and $D(\mu_\lambda,q)=\log 2/\log(1/\lambda)$ for all $q$). For $\lambda=1/2$, the Bernoulli convolution $\mu_\lambda$ is a multiple of Lebesgue measure on the interval $[0,1/(1-\lambda)]$. Understanding the smoothness properties of $\mu_\lambda$ for $\lambda\in (1/2,1)$ has been a major open problem since the 1930s. Although the problem is still very much open, dramatic progress has been achieved in the last few years. In this section we briefly state the consequences of the results of the previous sections for Bernoulli convolutions, and discuss their connections with other old and new results about them.

In two foundational papers, Erd\H{o}s \cite{Erdos39, Erdos40} showed that $\mu_\lambda$ is singular if $1/\lambda$ is a Pisot number (an algebraic integer $>1$ all of whose algebraic conjugates are $<1$ in modulus), and that $\mu_\lambda$ has a density in $C^k$ for almost all $\lambda$ sufficiently close to $1$ (depending on $k$). In the 1960s, Garsia \cite{Garsia62} exhibited an explicit infinite family of algebraic numbers $\lambda$ for which $\nu_\lambda$ is absolutely continuous. These remained the only explicit known parameters of absolute continuity until very recently when Varju \cite{Varju19}, introducing several new techniques, exhibited a new large family of algebraic numbers very close to $1$ for which $\mu_\lambda$ is absolutely continuous, with a density in $L\log L$.

In a celebrated paper, Solomyak \cite{Solomyak95} proved that $\mu_\lambda$ is absolutely continuous with an $L^2$ density for almost all $\lambda\in (1/2,1)$. Much more recently, in another landmark paper \cite{Hochman14} that we have already encountered several times, Hochman proved that $\hdim(\mu_\lambda)=1$ for all $\lambda$ outside of a set of $\lambda$ of zero Hausdorff (and even packing) dimension. Building on that, the author \cite{Shmerkin14} proved that $\mu_\lambda$ is absolutely continuous  for all $\lambda$ outside of a set of $\lambda$ of zero Hausdorff dimension. As an immediate application of Theorem \ref{thm:abs-continuity}, we have:
\begin{corollary}
There exists a set $E\subset (1/2,1)$ of zero Hausdorff dimension such that $\nu_\lambda$ is absolutely continuous and its density is in $L^q$ for all $q\in (1,\infty)$, for all $\lambda\in (1/2,1)\setminus E$.
\end{corollary}
We underline that the information that the density is in $L^q$ for $q>2$ is new even for a.e. parameter. Note that Corollary \ref{cor:Frostman-exp} shows that $\mu_\lambda$ has Frostman exponent $1-\e$ for every $\e>0$ for every $\lambda$ for which there is exponential separation. Although this is weaker than $L^q$ density for all $q>1$, exponential separation can be checked for some explicit parameters; in particular it holds for all rationals in $(1/2,1)$.

An active area of research concerns investigating the properties of $\mu_\lambda$ for algebraic values of $\lambda$. We only summarize some of the recent results in this area. The entropy of a purely atomic measure $\nu$ is defined as $H(\nu)=\sum_x \nu(x)\log(1/\nu(x))$. The \emph{Garsia entropy} associated to $\mu_\lambda$ is defined as
\[
h_\lambda = \lim_{n\to\infty} \frac{1}{n} H(\mu_{\lambda,n}),
\]
where $\mu_{\lambda,n}$ is the $n$-th step discrete approximation to $\mu_\lambda$, that is, the distribution of the finite random sum $\sum_{j=0}^{n-1} X_j \lambda^j$. It is well-known that the limit exists. The number $h_\lambda$ can also be interpreted as the entropy of the uniform random walk generated by the similarities $\lambda x$ and $\lambda x+1$.

It follows from Hochman's work \cite{Hochman14} (see \cite[\S 3.4]{BreuillardVarju18} for a detailed argument) that if $\lambda$ is algebraic, then
\begin{equation} \label{eq:Hocuman-Garsia-entropy}
\hdim(\mu_\lambda) = \min\left(\frac{h_\lambda}{\log(1/\lambda)},1\right).
\end{equation}
Breuillard and Varju \cite[Theorem 5]{BreuillardVarju18} gave bounds for $h_\lambda$ in terms of the Mahler measure $M_\lambda$ of $\lambda$ (see e.g. \cite[Eq. (1.1)]{BreuillardVarju18} for the definition of Mahler measure):
\begin{equation} \label{eq:breuillard-varju}
c\min(1,\log M_\lam) \le h_\lambda \le \min(1,\log M_\lam),
\end{equation}
where $c>0$ is a universal constant that they numerically estimate to be at least $0.44$. Using this theorem, they uncover a connection between Bernoulli convolutions and problems related to growth rates in linear groups. Very roughly, the idea is that the worst possible rate occurs for the group generated by the similarities $\lambda x$, $\lambda x+1$, which can be easily  realized as a linear group. An easy consequence of \eqref{eq:Hocuman-Garsia-entropy} and \eqref{eq:breuillard-varju} is that, assuming Lehmer's conjecture that the Mahler measure $M_\lam$ is either $1$ or bounded away from $1$, the Hausdorff dimension of $\mu_\lambda$ is $1$ for all algebraic numbers which are close enough to $1$. In \cite{BreuillardVarju19}, further progress was obtained; among many other results, the authors show that if $\hdim(\mu_\lambda)<1$ for some transcendental number $\lambda$, then $\lambda$ can be approximated by algebraic numbers with the same property. Hence, conditional on the Lehmer conjecture, $\hdim(\mu_\lambda)=1$ for all $\lambda$ close to $1$. Very recently, combining results from most of the papers mentioned in this section with a clever new argument, Varju \cite{Varju19b} achieved another major breakthrough by proving that $\hdim(\mu_\lambda)=1$ for all transcendental $\lambda\in (1/2,1)$.

The formula \eqref{eq:Hocuman-Garsia-entropy} makes it important to be able to compute Garsia entropy. An algorithm for this was developed in \cite{AFKP18}. Among other applications, this algorithm makes it possible to check that $\hdim(\mu_\lambda)=1$ for specific (new) algebraic values of $\lambda$.

All of these recent advances depend on the formula \eqref{eq:Hocuman-Garsia-entropy}, and hence apply only to Hausdorff dimension and not to $L^q$ dimensions. However, Corollary \ref{cor:algebraic} shows that the $L^q$ version of \eqref{eq:Hocuman-Garsia-entropy} remains valid: for all algebraic $\lambda\in (1/2,1)$,
\[
D(\mu_\lambda,q) =  \min\left( \frac{T_{q,\lambda}}{(q-1)\log(1/\lambda)},1\right),
\]
where
\[
T_{q,\lambda} = \lim_{n\to\infty }-\frac{1}{n} \log\|\mu_{\lambda,n}\|_q^q
\]
is an $L^q$ analog of Garsia entropy. Hence it would be interesting to know if there are $L^q$ versions of some of the results described above.

\subsection{Intersections of Cantor sets}

To finish the paper, we show how Theorem \ref{thm:main-homogeneous} can be used to obtain strong bounds on the dimensions of intersections of certain Cantor sets. Indeed, a conjecture of Furstenberg about the dimensions of intersections of $\times 2$, $\times 3$-invariant closed subsets of the circle was the main motivation for the results of \cite{Shmerkin19}. While the resolution of Furstenberg's intersection conjecture requires a more general version of Theorem \ref{thm:main-homogeneous} and is therefore beyond the scope of this survey, we will still be able to derive other intersection bounds.

In the following simple lemma we show how Frostman exponents (and therefore, by Lemma \ref{lem:Lq-dim-to-Frostman-exp}, also $L^q$ dimensions) of projected measures give information about the size of fibers. We recall the definition of upper box-counting (or Minkowski) dimension in a totally bounded metric space $(X,d)$. Given $A\subset X$, let $N_\e(A)$ denote the maximal cardinality of an $\e$-separated subset of $A$. The upper box-counting dimension of $A$ is then defined as
\[
\ubdim(A) = \limsup_{\e\downarrow 0} \frac{\log(N_\e(A))}{\log(1/\e)}.
\]
\begin{lemma} \label{lem:Frostman-exp-to-small-fiber}
Let $X$ be a compact metric space, and suppose $\pi:X\to \R$ is a Lipschitz map. Let $\mu$ be a probability measure on $X$ such that $\mu(B(x,r)) \ge r^s$ for all $x\in X$ and all sufficiently small $r$ (independent of $x$). If $\pi\mu$ has Frostman exponent $\alpha$, then there exists $C>0$ such that for all balls $B_\e$ of radius $\e$ in $\R$, any $\e$-separated subset of $\pi^{-1}(B_\e)$ has size at most $C\e^{-(s-\alpha)}$.

In particular, for any $y\in\R$,
\[
\ubdim(\pi^{-1}(y)) \le s-\alpha
\]
\end{lemma}
\begin{proof}
Let $(x_j)_{j=1}^M$ be an $\e$-separated subset of $\pi^{-1}(B_\e)$ with $\e$ small. Then
\[
\mu\left(\bigcup_{j=1}^M B(x_j,\e/2)\right) \ge M (\e/2)^s,
\]
while the set in question projects onto an interval of size at most $O(\e)$. Hence $M=O(\e^{\alpha-s})$, giving the claim.
\end{proof}

We give one concrete application of Theorem \ref{thm:main-homogeneous} in conjunction with this lemma, and refer to \cite[\S 6.3]{Shmerkin19} for further examples. Let $p\ge 2$ be an integer, and let $D\subset\{0,1,\ldots,p-1\}$ be a proper subset. Let $A=A_{p,D}$ be the set of $[0,1]$ consisting of all points whose $p$-ary expansion has only digits from $D$. This is the self-similar set associated to the IFS $((x+j)/p: j\in D)$. For example, the middle-thirds Cantor set is the case $p=3$, $D=\{0,2\}$. We call such a set a \emph{$p$-Cantor set}.

\begin{corollary} 
Let $A\subset [0,1)$ be a $p$-Cantor set, $p\ge 2$. Then for every irrational number $t\in\R$ and any $u\in\R$,
\[
\ubdim(A\cap (t A+ u)) \le \max(2\hdim(A)-1,0).
\]
\end{corollary}
\begin{proof}
Let $A=A_{p,D}$, and let $\mu$ be the uniform self-similar measure on $A$. Since the IFS generating $A$ satisfies the open set condition, it is well-known, and not hard to see, that $\mu(B(x,r)) = \Theta(r^s)$ for all $x\in A$, with the implicit constant depending only on $p,D$. Hence the product measure $\mu\times\mu$ satisfies
\begin{equation} \label{eq:product-AR}
(\mu\times\mu)(B(z,r)) = \Theta(r^{2s})
\end{equation}
for all $z\in A\times A=\supp(\mu\times\mu)$.

Let $\Pi_t(x,y)=x+ty$. Then $\Pi_t(\mu\times\mu)$ is the uniform self-similar measure generated by the IFS
\[
\left( p^{-1}(x+i+tj): i,j\in D\right).
\]
We claim that this IFS has exponential separation for all irrational $t$. Assuming the claim, the corollary follows by combining Corollary \ref{cor:Frostman-exp} and Lemma \ref{lem:Frostman-exp-to-small-fiber} (keeping \eqref{eq:product-AR} in mind) .

The argument to establish exponential separation in this setting is due to B. Solomyak and the author, and was originally featured in \cite[Theorem 1.6]{Hochman14}. Fix $t\in\R\setminus\mathbb{Q}$. The separation number $\Gamma_k$ associated to $\Pi_t\mu$ has the form $x_k +  t y_k$, where $x_k, y_k$ have the form $\sum_{j=0}^{k-1} a_j p^{-j}$ with $a_j\in D-D$. Moreover, $x_k$ and $y_k$ cannot be simultaneously $0$ since this would imply an exact overlap in the IFS generating $A$. If either $x_k$ or $y_k$ are zero for infinitely many $k$, then $\Gamma_k \ge \min(1,t)p^{-k}$ for infinitely many $k$ and hence we are done. So assume $x_k y_k\neq 0$ for all $k\ge k_0$, and therefore
\[
\left|\frac{\Gamma_k}{y_k}-\frac{\Gamma_{k+1}}{y_{k+1}}\right| = \left|\frac{x_k}{y_k}-\frac{x_{k+1}}{y_{k+1}}\right| = \left|\frac{z_k}{y_k y_{k+1}}\right|,
\]
where $z_k=x_k y_{k+1}-x_{k+1}y_k$. If $z_k=0$ for all $k\ge k_1$, then for all $k\ge k_1$ we have
\[
\Gamma_k = |y_k(x_{k_1}/y_{k_1}+t)| \ge p^{-k-1}|x_{k_1}/y_{k_1}+t|
\]
so, again using the irrationality of $t$, there is exponential separation. It remains to analyze the case $z_k\neq 0$ for infinitely many $k$. For any such $k$, the quotient $z_k/(y_k y_{k+1})$ is a non-zero rational number of denominator at most $4 p^{2k+1}$. Since $|y_k|\le 2$ for all $k$, we conclude that there are infinitely many $k$ such that either $\Gamma_k \ge p^{-2k-1}/16$ or $\Gamma_{k+1}\ge p^{-2k-1}/16$. Thus exponential separation also holds in this case, finishing the proof.
\end{proof}

For rational $t$ the behaviour is completely different: it follows from \cite[Theoren 1.2]{BaranyRams14} that if $A=A_{p,D}$ is any $p$-Cantor set of dimension $>1/2$, and $p\nmid |D|^2$ (in particular this holds if $p$ is prime), then for every rational $t$ there are many values of $u$ such that
\[
\hdim(A\cap tA+u) > 2\hdim(A)-1.
\]
More precisely, for a given $t$ this holds for a typical $u$ chosen according to the natural self-similar measure on $A\times A$.


\begin{thebibliography}{10}

\bibitem{ALL17}
Christoph Aistleitner, Gerhard Larcher, and Mark Lewko.
\newblock Additive energy and the {H}ausdorff dimension of the exceptional set
  in metric pair correlation problems.
\newblock {\em Israel J. Math.}, 222(1):463--485, 2017.
\newblock With an appendix by Jean Bourgain.

\bibitem{AFKP18}
Shigeki Akiyama, De-Jun Feng, Tom Kempton, and Tomas Persson.
\newblock On the {H}ausdorff dimension of {B}ernoulli convolutions.
\newblock {\em Int. Math. Res. Not. IMRN}, accepted for publication, 2018.
\newblock arXiv:1801.07118.

\bibitem{BaranyRams14}
Bal{\'a}zs B{\'a}r{\'a}ny and Micha{\l} Rams.
\newblock Dimension of slices of {S}ierpi\'nski-like carpets.
\newblock {\em J. Fractal Geom.}, 1(3):273--294, 2014.

\bibitem{BRS19}
BAL\'{A}ZS B\'{A}R\'{A}NY, MICHA\L RAMS, and K\'{A}ROLY SIMON.
\newblock On the dimension of triangular self-affine sets.
\newblock {\em Ergodic Theory Dynam. Systems}, 39(7):1751--1783, 2019.

\bibitem{BreuillardVarju18}
Emmanuel Breuillard and P\'{e}ter~P. Varj\'{u}.
\newblock Entropy of {B}ernoulli convolutions and uniform exponential growth
  for linear groups.
\newblock {\em J. Anal. Math.}, accepted for publication, 2018.
\newblock arXiv:1510.04043.

\bibitem{BreuillardVarju19}
Emmanuel Breuillard and P\'{e}ter~P. Varj\'{u}.
\newblock On the dimension of {B}ernoulli convolutions.
\newblock {\em Ann. Probab.}, 47(4):2582--2617, 2019.

\bibitem{DyatlovZahl16}
Semyon Dyatlov and Joshua Zahl.
\newblock Spectral gaps, additive energy, and a fractal uncertainty principle.
\newblock {\em Geom. Funct. Anal.}, 26(4):1011--1094, 2016.

\bibitem{Erdos39}
Paul Erd\H{o}s.
\newblock On a family of symmetric {B}ernoulli convolutions.
\newblock {\em Amer. J. Math.}, 61:974--976, 1939.

\bibitem{Erdos40}
Paul Erd\H{o}s.
\newblock On the smoothness properties of a family of {B}ernoulli convolutions.
\newblock {\em Amer. J. Math.}, 62:180--186, 1940.

\bibitem{FraserJordan17}
Jonathan~M. Fraser and Thomas Jordan.
\newblock The {A}ssouad dimension of self-affine carpets with no grid
  structure.
\newblock {\em Proc. Amer. Math. Soc.}, 145(11):4905--4918, 2017.

\bibitem{Garsia62}
Adriano~M. Garsia.
\newblock Arithmetic properties of {B}ernoulli convolutions.
\newblock {\em Trans. Amer. Math. Soc.}, 102:409--432, 1962.

\bibitem{Hochman14}
Michael Hochman.
\newblock On self-similar sets with overlaps and inverse theorems for entropy.
\newblock {\em Ann. of Math. (2)}, 180(2):773--822, 2014.

\bibitem{Hochman17}
Michael Hochman.
\newblock On self-similar sets with overlaps and inverse theorems for entropy
  in $\mathbb{R}^d$.
\newblock {\em Mem. Amer. Math. Soc.}, 2017.

\bibitem{LauNgai99}
Ka-Sing Lau and Sze-Man Ngai.
\newblock Multifractal measures and a weak separation condition.
\newblock {\em Adv. Math.}, 141(1):45--96, 1999.

\bibitem{PeresSolomyak00}
Yuval Peres and Boris Solomyak.
\newblock Existence of {$L\sp q$} dimensions and entropy dimension for
  self-conformal measures.
\newblock {\em Indiana Univ. Math. J.}, 49(4):1603--1621, 2000.

\bibitem{PeresSolomyak98}
Yuval Peres and Boris Solomyak.
\newblock Problems on self-similar sets and self-affine sets: an update.
\newblock In {\em Fractal geometry and stochastics, {II}
  ({G}reifswald/{K}oserow, 1998)}, volume~46 of {\em Progr. Probab.}, pages
  95--106. Birkh\"{a}user, Basel, 2000.

\bibitem{RossiShmerkin18}
Eino Rossi and Pablo Shmerkin.
\newblock On measures that improve ${L}^q$ dimension under convolution.
\newblock Preprint, arXiv:1812.05660, 2018.

\bibitem{Shmerkin14}
Pablo Shmerkin.
\newblock On the exceptional set for absolute continuity of {B}ernoulli
  convolutions.
\newblock {\em Geom. Funct. Anal.}, 24(3):946--958, 2014.

\bibitem{Shmerkin19}
Pablo Shmerkin.
\newblock On {F}urstenberg's intersection conjecture, self-similar measures,
  and the {$L^q$} norms of convolutions.
\newblock {\em Ann. of Math. (2)}, 189(2):319--391, 2019.

\bibitem{ShmerkinSolomyak16}
Pablo Shmerkin and Boris Solomyak.
\newblock Absolute continuity of self-similar measures, their projections and
  convolutions.
\newblock {\em Trans. Amer. Math. Soc.}, 368(7):5125--5151, 2016.

\bibitem{Solomyak95}
Boris Solomyak.
\newblock On the random series {$\sum\pm\lambda\sp n$} (an {E}rd{\H o}s
  problem).
\newblock {\em Ann. of Math. (2)}, 142(3):611--625, 1995.

\bibitem{Varju19b}
P\'{e}ter Varj\'{u}.
\newblock On the dimension of {B}ernoulli convolutions for all transcendental
  parameters.
\newblock {\em Ann. of Math. (2)}, 189(3):1001--1011, 2019.

\bibitem{Varju19}
P\'{e}ter~P. Varj\'{u}.
\newblock Absolute continuity of {B}ernoulli convolutions for algebraic
  parameters.
\newblock {\em J. Amer. Math. Soc.}, 32(2):351--397, 2019.

\end{thebibliography}

\end{document}